\documentclass{amsart}
\usepackage{amsfonts}
\usepackage{latexsym}
\usepackage{amssymb}
\usepackage{amsmath}
\usepackage{color}
\usepackage{bbm}
\usepackage{tikz}
\usepackage{enumerate}
\usepackage{mathrsfs}

\newcommand{\G}{\mathbb G}
\newcommand{\R}{\mathbb R}
\newcommand{\N}{\mathbb N}

\newcommand{\E}{\mathbb E}
\newcommand{\Pro}{\mathbb P}
\newcommand{\dif}{\,\mathrm{d}}

\def\dint{\textup{d}}
\newcommand{\SSS}{\ensuremath{{\mathbb S}}}

\newcommand{\B}{\ensuremath{{\mathbb B}}}
\newcommand{\bPEX}{\boldsymbol{\|P_EX\|}}

\newtheorem{thm}{Theorem}[section]
\newtheorem{cor}[thm]{Corollary}
\newtheorem{lemma}[thm]{Lemma}
\newtheorem{df}[thm]{Definition}
\newtheorem{proposition}[thm]{Proposition}

\newtheorem{rmk}[thm]{Remark}

\def\bG{\mathbf{G}}

\def\bU{\mathbf{U}}
\def\bV{\mathbf{V}}
\def\bW{\mathbf{W}}
\def\bX{\mathbf{X}}
\def\bY{\mathbf{Y}}
\def\bZ{\mathbf{Z}}

\usepackage{nomencl}
\makenomenclature

\allowdisplaybreaks

\begin{document}


\title[LDPs for projections of $\ell_p^n$-balls]{Large deviations for high-dimensional\\ random projections of $\ell_p^n$-balls }

\author[D. Alonso-Guti\'errez]{David Alonso-Guti\'errez}
\address{Departamento de Matem\'aticas, Universidad de Zaragoza, Spain} \email{alonsod@unizar.es}

\author[J. Prochno]{Joscha Prochno}
\address{School of Mathematics \& Physical Sciences, University of Hull, United Kingdom} \email{j.prochno@hull.ac.uk}

\author[C. Th\"ale]{Christoph Th\"ale}
\address{Faculty of Mathematics, Ruhr University Bochum, Germany} \email{christoph.thaele@rub.de}




\begin{abstract}
The paper provides a description of the large deviation behavior for the Euclidean norm of projections of $\ell_p^n$-balls to high-dimensional random subspaces. More precisely, for each integer $n\geq 1$, let $k_n\in\{1,\ldots,n-1\}$, $E^{(n)}$ be a uniform random $k_n$-dimensional subspace of $\R^n$ and $X^{(n)}$ be a random point that is uniformly distributed in the $\ell_p^n$-ball of $\R^n$ for some $p\in[1,\infty]$. Then the Euclidean norms $\|P_{E^{(n)}}X^{(n)}\|_2$ of the orthogonal projections are shown to satisfy a large deviation principle as the space dimension $n$ tends to infinity. Its speed and rate function are identified, making thereby visible how they depend on $p$ and the growth of the sequence of subspace dimensions $k_n$. As a key tool we prove a probabilistic representation of $\|P_{E^{(n)}}X^{(n)}\|_2$ which allows us to separate the influence of the parameter $p$ and the subspace dimension $k_n$.
\end{abstract}

\maketitle

\tableofcontents

\section{Introduction}

The geometry of convex bodies in high dimensions is a fascinating and vivid field at the core of what is known today as Asymptotic Geometric Analysis, a branch of mathematics at the crossroads between analysis, geometry and probability. In particular, it has been realized in the last decades that the presence of high dimensions forces certain regularity on the geometry of convex bodies that in many instances has a probabilistic flavor, compare with the surveys of Gu\'edon \cite{G114,G214} and the monograph \cite{GeometryICB}, for example. The arguably most prominent example is the central limit theorem, which is widely known in probability theory to capture the fluctuations of a sum of (independent) random variables (see, e.g., Chapter 5 in \cite{Kallenberg}). In the geometric context it roughly says that most $k$-dimensional marginals of a high-dimensional isotropic convex body are approximately Gaussian, provided that $k$ is of smaller order than $n^\kappa$ for some universal constant $\kappa\in(0,1)$, i.e., $k=o(n^\kappa)$. The central limit theorem for convex bodies was conjectured in \cite{ABP2003} by Anttila, Ball and Perissinaki (for $k=1$), who proved the conjecture for the case of uniform distributions on convex sets whose modulus of convexity and diameter satisfy some additional quantitative assumptions. Other contributions to different facets of the central limit problem for (special classes of) convex bodies are due to Bobkov and Koldobsky \cite{BobkovKoldobsky}, Brehm, Hinow, Vogt and Voigt \cite{BrehmEtAl}, E. Meckes \cite{Meckes,Meckes2}, E. and M. Meckes \cite{MeckesMeckes}, E. Milman \cite{MilmanE} or Paouris \cite{Paouris}, just to mention a few. For general bodies, based on a principle going back to the work of Sudakov \cite{S1978}, and Diaconis and Freedman \cite{DF1984}, a central limit theorem was proved by Klartag in \cite{KlartagCLT,KlartagCLT2}, who obtained that $\kappa\geq 1/15$. If in addition the convex body is $1$-unconditional, that is, symmetric with respect to all coordinate hyperplanes, this has been extended by M.\ Meckes \cite{MeckesM12} to $k$-dimensional marginals with $k=o(n^{1/3})$. In particular, this class of convex bodies includes the $\ell_p^n$-balls considered in the present text.

On the one hand the central limit theorem underlines the universal behavior of Gaussian fluctuations. On the other hand, it is widely known in probability theory that the so-called large deviation behavior, which considers fluctuations beyond the Gaussian scale, is much more sensitive to the distributions of the involved random variables. For example, Cram\'er's theorem (see, e.g., \cite[Theorem 2.2.3]{DZ} or \cite[Theorem 27.5]{Kallenberg}) guarantees that if $X,X_1,X_2,\ldots$ are independent, identically distributed and centered random variables with cumulant generating function $\Lambda(u):=\log(\E e^{uX})<\infty$ for all $u\in\R$, one has that
$$
\lim\limits_{n\to\infty}{1\over n}\log \Pro(X_1+\ldots+X_n\geq nt)=-\Lambda^*(t)
$$
for all $t\in\R$, where $\Lambda^*$ is the Legendre-Fenchel transform of $\Lambda$. Equivalently, this means that for any $\varepsilon>0$ and any $t\in\R$ there exists some natural number $n_0$ so that for each $n\geq n_0$,
\[
e^{-n(\Lambda^*(t)+\varepsilon)} \leq \Pro(X_1+\ldots+X_n\geq nt) \leq e^{-n(\Lambda^*(t)-\varepsilon)}.
\]
We emphasize that this function usually displays an entirely different behavior for random variables sharing the same properties on the scale of the central limit theorem. While large deviations have been investigated intensively in probability theory (see, for instance, \cite{DZ,dH} and the references cited therein), they have -- in sharp contrast to the central limit theorem -- left almost no traces in Asymptotic Geometric Analysis so far. However and as already anticipated above, the study of large deviations of marginals of high-dimensional convex bodies might open new perspectives and give access to non-universal features that allow to make transparent properties that distinguish between different convex bodies. In addition to the potential mentioned before, random projections of random vectors in high dimensions naturally appear in machine learning and information science, for instance, in linear regression \cite{MM2012} when searching for the best regression
function and for the purpose of dimension reduction in information retrieval in text documents \cite{BM2001} to reduce the computational complexity.

It was only recently that Gantert, Kim and Ramanan \cite{GKRCramer,GKR}, and Kim and Ramanan \cite{KimRamanan} opened this field by deriving, in particular, a Large Deviation Principle (LDP) in the spirit of Donsker and Varadhan for $1$-dimensional random projections of $\ell_p^n$-balls in $\R^n$, as the space dimension $n$ tends to infinity. More precisely, their results show that if for each $n\in\N$, $\Theta^{(n)}\in\SSS^{n-1}$ is a uniform random direction and $X^{(n)}$ is an independent random point uniformly distributed in the $\ell_p^n$-ball of $\R^n$ for some fixed $p\in[1,\infty]$, then the sequence of rescaled random variables
$$
n^{{1\over p}-{1\over 2}}\langle X^{(n)},\Theta^{(n)}\rangle
$$
satisfies an LDP with speed $n^{2p\over 2+p}$ if $p\in[1,2)$ and speed $n$ if $p\in[2,\infty]$ and with a certain rate function that also depends on $p$ (all notions and notation are explained in Section \ref{sec:Preliminaries} below). In view of Klartag's multi-dimensional version of the central limit theorem for convex bodies (see \cite[Theorem 1.3]{KlartagCLT} and \cite[Theorem 1.1]{KlartagCLT2}) it is also natural to consider projections onto  higher-dimensional random subspaces as well. The purpose of the present paper is to put the results from \cite{GKR} into a wider context and to provide a description of the large deviation behavior for the Euclidean norm of projections of $\ell_p^n$-balls onto random subspaces in high dimensions. At the same time it helps to clarify the r\^ole of the involved parameters.
The essential step to an extension to higher dimensions is our novel probabilistic representation of the Euclidean norm of a random projection (see Theorem \ref{RepresentationAnnealed}) that allows us to separate the influence of the parameter $p$ and the subspace dimension. This representation might be of independent interest.

Let us explain our main results in more detail (again, we refer to Section \ref{sec:Preliminaries} below for any unexplained notion or notation). We fix $p\in[1,\infty]$ and let for each $n\in\N$, $X^{(n)}$ be an independent random point that is uniformly distributed in the $\ell_p^n$-ball of $\R^n$. Furthermore, we let $k_n\in\{1,\ldots,n-1\}$ be an integer and assume that $E^{(n)}$ is a random subspace distributed according to the Haar probability measure on the Grassmann manifold of $k_n$-dimensional subspaces in $\R^n$ which is independent of $X^{(n)}$. The sequence of random variables of interest to us are the Euclidean norms of the orthogonal projections $P_{E^{(n)}}X^{(n)}$ of $X^{(n)}$ onto $E^{(n)}$, that is,
\[
\big\|P_{E^{(n)}}X^{(n)}\big\|_2\,,\qquad n\in\N.
\]
We set $\bPEX:=\big(n^{{1\over p}-{1\over 2}}\|P_{E^{(n)}}X^{(n)}\|_2\big)_{n\in\N}$ and note that if $k_n=1$ for all $n\in\N$, this reduces to the sequence of random variables studied in \cite{GKR}. We first consider the case $p\in[2,\infty]$ and define for $p\in[2,\infty)$ the function
\[
\mathcal{J}_p(y):=\inf_{x_1, x_2>0\atop {x_1^{1/2}}{x_2^{-{1/p}}}=y}\mathcal{I}_p^*(x_1,x_2)\,,\qquad y\in\R\,,
\]
where $\mathcal{I}_p^*(x_1,x_2)$ is the Legendre-Fenchel transform of
$$
\mathcal{I}_p(t_1,t_2):=\log\left(\int_{\R}e^{t_1x^2+t_2\vert x\vert^p}\,f_p(x)\,\dint x\right)\,,\quad (t_1,t_2)\in\R\times\Big(-\infty,{1\over p}\Big)
$$
with $f_p(x):=(2p^{1/p}\Gamma(1+1/p))^{-1}e^{-|x|^p/p}$, $x\in\R$, being the density of a $p$-generalized Gaussian random variable. To handle the exceptional case $p=\infty$ simultaneously, we write $\mathcal{J}_\infty(y):=\mathcal{I}_\infty^*(y^2)$ with $\mathcal{I}_\infty^*$ being the Legendre-Fenchel transform of $\mathcal{I}_\infty(t):=\log\big(2\int_0^1e^{tx^2}\,\dint x\big)$. Our first main result reads as follows.

\begin{thm}\label{thm:p>2}
Let $p\in[2,\infty]$ and assume that the limit $\lambda:=\lim\limits_{n\to\infty}{k_n\over n}$ exists in $[0,1]$.  Then the sequence $\bPEX$ satisfies an LDP with speed $n$ and rate function
$$
\mathcal{I}_{\bPEX}(y) := \begin{cases}
\inf\limits_{x\geq y}\Big[{\lambda\over 2}\log\big({\lambda x^2\over y^2}\big)+{1-\lambda\over 2}\log\big({1-\lambda\over 1-y^2x^{-2}}\big)+\mathcal{J}_p(x)\Big] &: y>0\\
\mathcal{J}_p(0) &: y=0,\, \lambda\in(0,1]\\
\inf\limits_{x\geq 0}\mathcal{J}_p(x) &: y=0,\, \lambda=0\\
+\infty &: y<0\,,
\end{cases}
$$
where we understand the cases $\lambda\in\{0,1\}$ as the corresponding limits.
\end{thm}

We emphasize at this point that while the LDP in Theorem \ref{thm:p>2} shows a universal speed, its rate function depends in a subtle way on the underlying convex body via the parameter $p$.

Next, we shall discuss the special case $p=2$, which corresponds to the Euclidean unit ball, in some more detail. First of all, in this situation the rate function can be made fully explicit and is given by
\begin{equation}\label{eq:RateFunctionp=2}
\mathcal{I}_{\bPEX}(y)=\begin{cases}
{\lambda\over 2}\log\big(\tfrac{\lambda}{y^2}\big)+{1-\lambda\over 2}\log\big(\tfrac{1-\lambda}{1-y^2}\big) &: y\in(0,1)\\
+\infty &: \text{otherwise}\,,
\end{cases}
\end{equation}
where we understand the cases $\lambda\in\{0,1\}$ as the corresponding limits and with $0$ or $1$ included in the effective domain of $\mathcal{I}_{\bPEX}$. In particular, if $\lambda$ takes the value zero the rate function reduces to
\[
\mathcal{I}_{\bPEX}(y)=\begin{cases}
-\frac{1}{2}\log(1-y^2) &: y\in[0,1)\\
+\infty &: \text{otherwise}
\end{cases}
\]
and this is exactly the rate function that already appeared in the $1$-dimensional LDP in \cite[Theorem 3.4]{BartheGamboaEtAl} or \cite[Theorem 2.12]{GKR}. In other words, this means that in the Euclidean case $p=2$ the LDP does not `feel' the random subspaces $E^{(n)}$ we project onto as long as their dimension is growing slowly with $n$, that is, if $k_n=o(n)$. The difference to the $1$-dimensional projections becomes visible only in the `truly' high-dimensional regime in which $k_n$ is eventually proportional to $n$.

We now turn to the case $p\in[1,2)$, which already for the $1$-dimensional projections shows a large deviation behavior at different scales, but this time with a fully explicit rate function (see \cite[Theorem 2.3]{GKR}). Our next results shows that this continues to hold for high-dimensional random projections as well.


\begin{thm}\label{thm:p<2}
Let $p\in[1,2)$ and assume that the limit $\lambda:=\lim\limits_{n\to\infty}{k_n\over n}$ exists in $(0,1]$.
Then the sequence $\bPEX$ satisfies an LDP with speed $n^{p/2}$ and rate function
$$
\mathcal{I}_{\bPEX}(y):=\begin{cases}
{1\over p}\big({y^2\over\lambda}-m\big)^{p\over 2} &: y\geq \sqrt{\lambda\, m}\\
+\infty &: \text{otherwise}\,,
\end{cases}
$$
where $m=m_p:={p^{p/2}\over 3}{\Gamma(1+{3\over p})\over\Gamma(1+{1\over p})}$.
\end{thm}

We emphasize that for $p\in[1,2)$ the LDP for random projections of $\ell_p^n$-balls holds at a non-universal and $p$-dependent speed. Moreover, a comparison with \cite[Theorem 2.3]{GKR} shows that both, the speed and the rate function differ from those for the $1$-dimensional random projections. In fact, in this situation (where $k_n = 1$ for all $n\in\N$) the sequence $\bPEX$ satisfies an LDP with speed $n^{2p\over 2+p}$ and rate function
$$
\mathcal{I}_{\bPEX}(y) = \begin{cases}
{2+p\over 2p}y^{2p\over 2+p} &: y\geq 0\\
+\infty &: \text{otherwise}
\end{cases}
$$
Note that the rate function stated here slightly differs from the rate function in \cite{GKR}, since we are not dealing with signed distances in our set-up but rather with their absolute values. Note that our Theorem \ref{thm:p<2} leaves open the case where the subspace dimensions $k_n$ are such that ${k_n\over n}\to 0$, as $n\to\infty$. We \textit{conjecture} that in this case the LDP for $\bPEX$ is the same as for $\bPEX$ with $k_n\equiv 1$ discussed above.


\medskip

After having presented our main theorems, let us comment on the tools we are going to use in their proofs. They basically reflect a lively interplay between geometric arguments with techniques and methods from large deviation theory. As already anticipated above, the key to Theorem \ref{thm:p>2} and Theorem \ref{thm:p<2} is a new probabilistic representation of the random variables $\|P_{E^{(n)}}X^{(n)}\|_2$. Notably, in the special case that $k_n=1$ for all $n\in\N$ this is different from the one that has been used in \cite{GKR}. More precisely, for each $n\in\N$ we will identify $\|P_{E^{(n)}}X^{(n)}\|_2$ with the product of three independent random variables:
$$
\|P_{E^{(n)}}X^{(n)}\|_2 = U^{1/n}\cdot Z^{(n)}\cdot G^{(n)}\,.
$$
Here,
\renewcommand\labelitemi{\tiny$\bullet$}
\begin{itemize}
\item $U$ is uniformly distributed on $[0,1]$,
\item $Z^{(n)}$ is the quotient of the $\ell_2^n$- and the $\ell_p^n$-norm of an $n$-dimensional random vector consisting of independent $p$-generalized Gaussian random entries,
\item $G^{(n)}$ is given by $(\sum_{i=1}^{k_n}g_i^2)^{1/2}/(\sum_{i=1}^{n}g_i^2)^{1/2}$ with standard Gaussian random variables $g_1,\ldots,g_n$ that are independent.
\end{itemize}
The essential feature of this representation is that the parameter $p$ influences only the random variables $Z^{(n)}$, while on the other hand the dimension parameter $k_n$ shows up exclusively in the definition of $G^{(n)}$. This in turn allows us to study the different effects separately and paves the way to the higher-dimensional generalizations of the results in \cite{GKR}. We emphasize that the representation of $\|P_{E^{(n)}}X^{(n)}\|_2$ as a product is well reflected by the rate function appearing in Theorem \ref{thm:p>2}, which possesses the following probabilistic interpretation: while the radial part $U^{1/n}$ has no influence as already seen in the $1$-dimensional case, the rate function is the infimum of the sum of two rate functions corresponding to LDPs for $Z^{(n)}$ and $G^{(n)}$. Moreover, the latter corresponds to the rate function \eqref{eq:RateFunctionp=2} appearing in the particular Euclidean case $p=2$.

\medskip

The rest of this paper is structured as follows. In Section \ref{sec:Preliminaries} we introduce our notation, recall the necessary background material from large deviation theory and provide some preliminaries on the geometry of $\ell_p^n$-balls. The aforementioned probabilistic representation of $\|P_{E^{(n)}}X^{(n)}\|_2$ is the content of Section \ref{sec:probabilistic representation}. We prove some auxiliary LDPs in Section \ref{sec:proof further LDPs} and in the final Section \ref{sec:proof main results} we eventually prove Theorem \ref{thm:p>2} and Theorem \ref{thm:p<2}. Since we have in mind a broad readership we decided to include background material, tools and arguments from both Asymptotic Geometric Analysis and probability theory.

\section{Preliminaries}\label{sec:Preliminaries}

\subsection{Notation}

In this paper we denote by $|A|$ the $n$-dimensional Lebesgue measure of a Lebesgue measurable set $A\subset\R^n$ and we write $\mathscr L(\R^n)$ for the $\sigma$-field of all Lebesgue measurable subsets of $\R^n$. The collection of Borel sets in $\R^n$ is denoted by $\mathscr B(\R^n)$. We supply the $n$-dimensional Euclidean space $\R^n$ with its standard inner product $\langle\,\cdot\,,\,\cdot\,\rangle$ and the Euclidean norm $\|\,\cdot\,\|_2$. The interior and the closure of a set $A\subset\R^n$ are denoted by $A^\circ$ and $\bar{A}$, respectively.

We write $\B_2^n:=\{x\in\R^n:\|x\|_2\leq 1\}$ for the Euclidean unit ball and $\SSS^{n-1}:=\{x\in\R^n:\|x\|_2=1\}$ for the corresponding unit sphere in $\R^n$, and $\sigma_{n-1}$ for the uniform probability measure on $\SSS^{n-1}$, that is, the normalized spherical Lebesgue measure. As subsets of $\R^n$ they carry natural Borel $\sigma$-fields that we denote by $\mathscr B(\B_2^n)$ and $\mathscr B(\SSS^{n-1})$, respectively. Moreover, we recall that
\begin{align}\label{eq:VolBall}
|\B_2^n| = {\pi^{n/2}\over\Gamma(1+{n\over 2})}\,,
\end{align}
where $\Gamma(\,\cdot\,)$ is the Gamma-function.

The group of $(n\times n)$-orthogonal matrices is denoted by $\mathcal O(n)$ and we let $\mathcal{SO}(n)$ be the subgroup of orthogonal $n\times n$ matrices with determinant $1$. As subsets of $\R^{n^2}$, $\mathcal{O}(n)$ and $\mathcal{SO}(n)$ can be equipped with the trace $\sigma$-field of $\mathscr B(\R^{n^2})$. Moreover, both compact groups $\mathcal O(n)$ and $\mathcal{SO}(n)$ carry a unique Haar probability measure which we denote by $\nu$ and $\widetilde{\nu}$, respectively. Since $\mathcal{O}(n)$ consists of two copies of $\mathcal{SO}(n)$, the measure $\nu$ can easily be derived from $\widetilde{\nu}$ and vice versa.

Given $k\in\{0,1,\ldots,n\}$, we use the symbol $\G_{n,k}$ to denote the Grassmannian of $k$-dimensional linear subspaces of $\R^n$. Denoting by $d_H(\,\cdot\,,\,\cdot\,)$ the Hausdorff distance we supply $\G_{n,k}$ with the metric $d(E,F):=d_H(B_E,B_F)$, $E,F\in \G_{n,k}$, where $B_E$ and $B_F$ stand for the Euclidean unit balls in $E$ and $F$, respectively. The Borel $\sigma$-field on $\G_{n,k}$ induced by this metric is denoted by $\mathscr B(\G_{n,k})$ and we supply the arising measurable space $\G_{n,k}$ with the unique Haar probability measure $\nu_{n,k}$. It can be identified with the image measure of the Haar probability measure $\widetilde{\nu}$ on $\mathcal{SO}(n)$ under the mapping $\mathcal{SO}(n)\to\G_{n,k},\, T\mapsto TE_0$ with $E_0:={\rm span}(\{e_1,\ldots,e_k\})$. Here, we write $e_1:=(1,0,\ldots,0),e_2:=(0,1,0,\ldots,0),\ldots,e_n:=(0,\ldots,0,1)\in\R^n$ for the standard orthonormal basis in $\R^n$ and ${\rm span}(\{e_1,\ldots,e_k\})\in \G_{n,k}$, $k\in\{1,\ldots,n\}$, for the $k$-dimensional linear subspace spanned by the first $k$ vectors of this basis.

\subsection{Large Deviation Principles}

The purpose of this section is to provide the necessary background material from large deviation theory, which may be found in \cite{DZ,dH,Kallenberg}, for example. We directly start with the definition of what we understand by a full and a weak large deviation principle. We refrain from presenting these definitions in the most general possible framework and rather restrict to the set-up needed in this paper. For this reason, let $d\geq 1$ be a fixed integer and assume that the $d$-dimensional Euclidean space $\R^d$ is supplied with its standard topology. In this subsection we denote for clarity the space dimension by $d$ instead of $n$ in order to distinguish it from our index parameter $n$. Finally, we make the assumption that all random objects we are dealing with are defined on a common (and sufficiently rich) probability space $(\Omega,\mathcal F,\Pro)$.

\begin{df}
Let $\bX:=(X^{(n)})_{n\in\N}$ be a sequence of random vectors taking values in $\R^d$. Further, let $s:\N\to[0,\infty]$ and $\mathcal{I}:\R^d\to[0,\infty]$ be a lower semi-continuous function with compact level sets $\{x\in\R^d\,:\, \mathcal{I}(x) \leq \alpha \}$, $\alpha\in\R$. We say that $\bX$ satisfies a (full) {\em large deviation principle} with speed $s(n)$ and (good) rate function $\mathcal{I}$ if
\begin{equation}\label{eq:LDPdefinition}
\begin{split}
-\inf_{x\in A^\circ}\mathcal{I}(x) &\leq\liminf_{n\to\infty}{1\over s(n)}\log(\Pro(X^{(n)}\in A))\\
&\leq\limsup_{n\to\infty}{1\over s(n)}\log(\Pro(X^{(n)}\in A))\leq-\inf_{x\in\overline{A}}\mathcal{I}(x)
\end{split}
\end{equation}
for all $A\in\mathscr L(\R^d)$. Moreover, we say that $\bX$ satisfies a {\rm weak large deviation principle} with speed $s(n)$ and rate function $\mathcal{I}$ if the lower bound in \eqref{eq:LDPdefinition} holds as stated, while the upper bound is valid only for compact sets $A\subset\R^n$.
\end{df}

We notice that on the class of all $\mathcal{I}$-continuity sets, that is, on the class of sets $A\in\mathscr L(\R^d)$ for which $\mathcal{I}(A^\circ)=\mathcal{I}(\bar{A})$ with $\mathcal{I}(A):=\inf\{\mathcal{I}(x):x\in A\}$, one has the exact limit relation
$$
\lim_{n\to\infty}{1\over s(n)}\log(\Pro(X^{(n)}\in A))=-\mathcal{I}(A)\,.
$$

In our paper we use the convention that the rate function in an LDP for a sequence of random vectors $\bX$ is denoted by $\mathcal{I}_\bX$.

What separates a weak from a full LDP is the so-called exponential tightness of the sequence of random variables (see, for instance, \cite[Lemma 1.2.18]{DZ} and \cite[Lemma 27.9]{Kallenberg}).

\begin{proposition}\label{prop:equivalence weak and full LDP}
Let $\bX:=(X^{(n)})_{n\in\N}$ be a sequence of random vectors taking values in $\R^d$. Suppose that $\bX$ satisfies a weak LDP with speed $s(n)$ and rate function $\mathcal{I}_\bX$. Then $\bX$ satisfies a full LDP if and only if $\bX$ is {\rm exponentially tight}, that is, if and only if
$$
\inf_K\limsup_{n\to\infty}{1\over s(n)}\log(\Pro(X^{(n)}\notin K))=-\infty\,,
$$
where the infimum is running over all compact sets $K\subset\R^d$.
\end{proposition}

The following proposition (see, for instance, \cite[Theorem 4.1.11]{DZ}) shows that it is sufficient to prove a weak LDP for a sequence of random variables solely for sets in a basis of the underlying topological space.

\begin{proposition}\label{prop:basis topology}
Let $d\in\N$ and $\mathcal A$ be basis of the standard topology in $\R^d$. Let $\bX=(X^{(n)})_{n\in\N}$ be a sequence of $\R^d$-valued random vectors. For every $A\in\mathcal A$, define
\[
\mathcal{I}_\bX^{(A)} := - \liminf_{n\to\infty}\frac{1}{s(n)}\log(\Pro(X^{(n)}\in A))
\]
and for $x\in\R^d$ set $\mathcal{I}_\bX(x):=\sup\{\mathcal{I}_\bX^{(A)}:A\in\mathcal{A},x\in A\}$. Suppose that for all $x\in\R^d$,
\[
\mathcal{I}_\bX(x) = \sup_{A\in\mathcal{A}\atop x\in A}\Big[-\limsup_{n\to\infty}\frac{1}{s(n)}\log\left(\Pro\left(X\in A\right)\right)\Big]\,.
\]
Then $\bX$ satisfies a weak LDP with speed $s(n)$ and rate function $\mathcal{I}_\bX$.
\end{proposition}

Let $d\geq 1$ be a fixed integer and let $X$ be an $\R^d$-valued random vector. We write
$$
\Lambda(u)=\Lambda_X(u):=\log (\E e^{\langle X,u\rangle})\,,\qquad u\in\R^d\,,
$$
for the cumulant generating function of $X$. Moreover, we define the (effective) domain of $\Lambda$ to be the set $D_{\Lambda}:=\{u\in\R^d:\Lambda(u)<\infty\}\subset\R^d$.

\begin{df}
The \textit{Legendre-Fenchel transform} of a convex function $\Lambda:\R^d \to(-\infty,+\infty]$ is defined as
$$
\Lambda^*(x):=\sup_{u\in\R^d}[\langle u, x\rangle -\Lambda(u)]\,,\qquad x\in\R^d\,.
$$
\end{df}

The Legendre-Fenchel transform of the cumulant generating function plays a crucial r\^ole in the following result, usually referred to as Cram\'er's theorem, (see, e.g., \cite[Theorem 2.2.30, Theorem 6.1.3, Corollary 6.1.6]{DZ} or \cite[Theorem 27.5]{Kallenberg}).

\begin{proposition}[Cram\'er's theorem]\label{prop:cramer}
Let $X,X_1,X_2,\ldots$ be independent and identically distributed random vectors taking values in $\R^d$. Assume that the origin is an interior point of $D_\Lambda$, where $\Lambda$ stands for the cumulant generating function of $X$. Then the partial sums ${1\over n}\sum\limits_{i=1}^n X_i$, $n\in\N$ satisfy an LDP with speed $n$ and good rate function $\Lambda^*$.
\end{proposition}

It will be rather important for us to deduce from an already existing large deviation principle a new one by applying various transformations. We first consider the large deviation behavior under the formation of vectors. For this, assume that $d_1$ and $d_2$ are integers and that $\bX=(X^{(n)})_{n\in\N}$ is a sequence of $\R^{d_1}$-valued random vectors and that $\bY=(Y^{(n)})_{n\in\N}$ is a sequence of $\R^{d_2}$-random vectors. Assuming that $\bX$ and $\bY$ satisfy large deviation principles, does then also the sequence $\bZ:=((X^{(n)},Y^{(n)}))_{n\in\N}$ of $\R^{d_1+d_2}$-valued random vectors satisfy a large deviation principle and, if so, what is its rate function? The following result is only implicit in \cite{DZ}. For the sake of completeness we present a self-contained proof in the appendix, since we were not able to precisely locate it in the existing literature.

\begin{proposition}\label{JointRateFunction}
Assume that $\bX$ satisfies an LDP with speed $s(n)$ and good rate function $\mathcal{I}_{\bX}$ and that $\bY$ satisfies an LDP with speed $s(n)$ and good rate function $\mathcal{I}_\bY$. Then, if $X^{(n)}$ and $Y^{(n)}$ are independent for every $n\in\N$, $\bZ=((X^{(n)},Y^{(n)}))_{n\in\N}$ satisfies an LDP with speed $s(n)$ and rate function $\mathcal{I}_\bZ$, where $\mathcal{I}_\bZ(x)=\mathcal{I}_\bX(x_1)+\mathcal{I}_\bY(x_2)$ for all $x=(x_1,x_2)\in\R^{d_1}\times\R^{d_2}$.
\end{proposition}

Next, assume that a sequence $\bX=(X^{(n)})_{n\in\N}$ of random variables satisfies an LDP with speed $n$ and rate function $\mathcal{I}_\bX$. Suppose now that $\bY=(Y^{(n)})_{n\in\N}$ is a sequence of random variables that are `close' to the ones from $\bX$. Our aim is to transfer in such a situation the LDP from $\bX$ to $\bY$. The conditions under which such an approach is working are the content of the next result, which we took from \cite[Theorem 4.2.13]{DZ} or \cite[Lemma 27.13]{Kallenberg}.

\begin{proposition}\label{prop:exponentially equivalent}
Let $\bX=(X^{(n)})_{n\in\N}$ and $\bY=(Y^{(n)})_{n\in\N}$ be two sequence of $\R^d$-valued random vectors and assume that $\bX$ satisfies an LDP with speed $s(n)$ and rate function $\mathcal{I}_\bX$. Further, suppose that $\bX$ and $\bY$ are {\rm exponentially equivalent}, i.e.,
$$
\limsup_{n\to\infty}{1\over s(n)}\log(\Pro(\|X^{(n)}-Y^{(n)}\|_2>\delta)) = -\infty
$$
for any $\delta>0$. Then $\bY$ satisfies an LDP with the same speed and the same rate function.
\end{proposition}

\begin{rmk}
If the dimension $d\in\N$ is fixed, then, since all norms are equivalent, we may consider the $\ell_1$-norm instead of the $\ell_2$-norm in the definition of exponential equivalence.
\end{rmk}

Finally, we consider the possibility to `transport' a large deviation principle to another one by means of a continuous function. This device is known as the contraction principle and we refer to \cite[Theorem 4.2.1]{DZ} or \cite[Theorem 27.11(i)]{Kallenberg}.

\begin{proposition}[Contraction principle]\label{prop:contraction principle}
Let $d_1,d_2\in\N$ and let $F:\R^{d_1}\to\R^{d_2}$ be a continuous function. Further, let $\bX=(X^{(n)})_{n\in\N}$ be a sequence of $\R^{d_1}$-valued random vectors that satisfies an LDP with speed $s(n)$ and rate function $\mathcal{I}_\bX$. Then the sequence $\bY:=(F(X^{(n)}))_{n\in\N}$ of $\R^{d_2}$-valued random vectors satisfies an LDP with the same speed and with rate function $\mathcal{I}_\bY=\mathcal{I}_\bX\circ F^{-1}$, i.e., $\mathcal{I}_\bY(y):=\inf\{\mathcal{I}_\bX(x):F(x)=y\}$, $y\in\R^{d_2}$, with the convention that $\mathcal{I}_\bY(y)=+\infty$ if $F^{-1}(\{y\})=\emptyset$.
\end{proposition}

While this form of the contraction principle was sufficient to analyse the large deviation behavior for $1$-dimensional random projections of $\ell_p^n$-balls, we will need a refinement to treat the higher-dimensional cases. More precisely, to handle this situation we need to allow the continuous function to depend on $n$. The following result can be found in \cite[Corollary 4.2.21]{DZ}.

\begin{proposition}\label{prop:refinement contraction principle}
Let $d_1,d_2\in\N$ and let $F:\R^{d_1}\to\R^{d_2}$ be a continuous function. Suppose that $\bX=(X^{(n)})_{n\in\N}$ is a sequence of $\R^{d_1}$-valued random variables that satisfies an LDP with speed $s(n)$ and rate function $\mathcal{I}_\bX$. Further, suppose that for each $n\in\N$, $F_n:\R^{d_1}\to\R^{d_2}$ is a measurable function such that for all $\delta>0$, $\Gamma_{n,\delta}:=\{x\in\R^{d_1}:\|F_n(x)-F(x)\|_2>\delta\}\in\mathscr L(\R^{d_1})$ and
$$
\limsup_{n\to\infty}{1\over s(n)}\log(\Pro(X^{(n)}\in\Gamma_{n,\delta}))=-\infty\,.
$$
Then the sequence of $\R^{d_2}$-valued random variables $(F_n(X^{(n)}))_{n\in\N}$ satisfies an LDP with the same speed and with rate function $\mathcal{I}_\bX\circ F^{-1}$.
\end{proposition}

\subsection{Geometry of $\ell_p^n$-balls}

Let $n\geq 1$ be an integer and consider the $n$-dimensio\-nal Euclidean space $\R^n$. For any $p\in[1,\infty]$ the $\ell_p^n$-norm, $\|x\|_p$, of $x=(x_1,\ldots,n)\in\R^n$ is given by
$$
\|x\|_p := \begin{cases}
\Big(\sum\limits_{i=1}^n|x_i|^p\Big)^{1/p} &: p<\infty\\
\max\{|x_1|,\ldots,|x_n|\} &: p=\infty\,.
\end{cases}
$$
Although $\|x\|_p$ depends on the space dimension $n$, we decided to suppress this dependency in our notation for simplicity, since $n$ will always be clear from the context.

For any $n$ and $p$ let us denote by $\B_p^n:=\{x\in\R^n:\|x\|_p\leq 1\}$ the $\ell_p^n$-ball in $\R^n$ and denote by $\SSS_p^{n-1}:=\{x\in\R^n:\|x\|_p=1\}$ the corresponding unit sphere. The restriction of the Lebesgue measure to $\B_p^n$ provides a natural volume measure on $\B_p^n$. Although one could supply $\SSS_p^{n-1}$ with the $(n-1)$-dimensional Hausdorff measure, the so-called cone measure turns out to be more useful as explained later (see \cite{NaorTAMS} for the relation between these two measures).

\begin{df}
For a set $A\in\mathscr{B}(\SSS_p^{n-1})$ we define
$$
\mu_p(A) := \frac{|\{rx:x\in A,r\in[0,1]\}|}{|\B_p^n|}\,.
$$
The measure $\mu_p$ is called the {\rm cone (probability) measure} of $\B_p^n$.
\end{df}

We remark that the cone measure $\mu_p$ coincides with the $(n-1)$-dimensional Hausdorff probability measure on $\SSS_p^{n-1}$ if and only if $p=1$, $p=2$ or $p=+\infty$. In particular, $\mu_2$ is the same as $\sigma_{n-1}$, the normalized spherical Lebesgue measure.

The proofs of our results heavily rely on the following probabilistic representations for the volume and the cone probability measure of $\B_p^n$ for $p\in[1,\infty)$, which are taken from \cite{RachevRueschendorf} and \cite{SchechtmanZinn} (we also refer to \cite{BartheGuedonEtAl} for a different representation).

\begin{proposition}\label{prop:schechtman zinn}
Let $n\in\N$ and $p\in[1,\infty)$. Suppose that $Z_1,\ldots,Z_n$ are independent $p$-generalized Gaussian random variables whose distribution has density
$$
f_p(x):= {1\over 2p^{1/p}\Gamma(1+{1\over p})}\,e^{-|x|^p/p}
$$
with respect to the Lebesgue measure on $\R$. Consider the random vector $Z:=(Z_1,\ldots,Z_n)\in\R^n$ and define $X:=Z/\|Z\|_p$. Furthermore, let $U$ be a uniformly distributed random variable on $[0,1]$, which is independent of the $Z_i$'s, and let us write $Y:=U^{1/n}X$. Then,
\begin{itemize}
\item[(i)] the random vector $X\in\SSS_p^{n-1}$ is independent of $\|Z\|_p$ and is distributed according to $\mu_p$,
\item[(ii)] the random vector $Y\in\B_p^n$ is uniformly distributed in $\B_p^n$.
\end{itemize}
\end{proposition}

In the rest of this paper $(g_i)_{i\in\N}$ will always denote a sequence of independent real-valued standard Gaussians, $U$ will denote an independent random variable uniformly distributed on $[0,1]$ and, for $p \geq 1$, $(Z_i)_{i\in\N}$ will denote a sequence of independent $p$-generalized Gaussian random variables with density $f_p$. All these random variables are assumed to be independent.

For further probabilistic aspects pertaining the geometry of $\ell_p^n$-balls we refer to \cite{BartheGuedonEtAl,SS1991,SchechtmanZinn,SZ2000} as well as the references cited therein.

\section{A probabilistic representation for $\|P_EX\|_2$}\label{sec:probabilistic representation}

In this section the dimension of the space $n$ will be fixed. Thus, for simplicity in the notation, we will omit the indices that will refer to the dimension $n$. Fix $p\in[1,\infty)$, let $X$ be a point chosen according to the uniform distribution on $\B_p^n$ and let $E\in \G_{n,k}$ be an independent random subspace with distribution $\nu_{n,k}$ for some $k\in\{1,\ldots,n\}$. In this section we will develop the already announced probabilistic representation for $\|P_EX\|_2$, which will turn out to be crucial in the proofs of Theorems \ref{thm:p>2} and \ref{thm:p<2}. The key feature of this representation is that it will allow us to identify $\|P_EX\|_2$ with a continuous function of two random variables $V_1^{(n)}$ and $W^{(n)}$. These random variables in turn can be written as functions of sums of independent identically distributed random variables. Besides, only one of them will depend on $p$, while the other one will depend only on the dimension $k$ of the random subspace $E$. These properties, together with Cram\'er's theorem and the contraction principle will give us the LDPs in the main theorems.

\begin{thm}\label{RepresentationAnnealed}
For any $n\in\N$ and $k\in\{1,\ldots,n\}$ let $X$ be a random vector uniformly distributed in $\B_p^n$ for some $p\in[1,\infty)$ and let $E\in \G_{n,k}$ be a random subspace distributed according to $\nu_{n,k}$. Then the random  variable $\| P_E X\|_2$ has the same distribution as the random variable
\[
U^{1/n}\frac{\left(\sum_{i=1}^n Z_i^2\right)^{1/2}}{\left(\sum_{i=1}^n \vert Z_i\vert ^p \right)^{1/p}}\frac{\left( \sum_{i=1}^k g_i^2\right)^{1/2}}{\left( \sum_{i=1}^n g_i^2\right)^{1/2}}\,.
\]
\end{thm}

\begin{proof}
Let $x\in\R^n$ be a fixed vector. By construction of the Haar measure $\nu_{n,k}$ on $\G_{n,k}$ and uniqueness of the Haar measure $\nu$ on $\mathcal{O}(n)$, we have that, for any $t\in\R$,
\begin{align*}
\nu_{n,k}\left(E\in \G_{n,k}\,:\,\| P_E x\|_2\geq t\right)&=\nu\left(T\in \mathcal O(n)\,:\,\| P_{T E_0} x\|_2\geq t\right)\cr
&=\nu\left(T\in \mathcal O(n)\,:\,\| P_{ E_0} Tx\|_2\geq t\right)\cr
&=\nu\bigg(T\in \mathcal O(n)\,:\,\| x\|_2\,\Big\| P_{ E_0} T\frac{x}{\| x\|_2}\Big\|_2\geq t\bigg)\,,
\end{align*}
where $E_0:=\textrm{span}(\{e_1,\dots,e_k\})$. Again, by the uniqueness of the Haar measure $\sigma_{n-1}$ on $\SSS^{n-1}$, $T(x/\| x\|_2)$ is a random vector uniformly distributed on $\SSS^{n-1}$ according to $\sigma_{n-1}$, provided that $T\in\mathcal{O}(n)$ has distribution $\nu$. Thus,
\[
\nu\bigg(T\in \mathcal O(n)\,:\,\| x\|_2\Big\| P_{ E_0} T\frac{x}{\| x\|_2}\,\Big\|_2\geq t\bigg)=\sigma_{n-1}\left(u\in S^{n-1}\,:\,\| x\|_2\left\| P_{ E_0} u\right\|_2\geq t\right)\,.
\]
Since $G=(g_1,\dots g_n)$ is a standard Gaussian random vector in $\R^n$, by Proposition \ref{prop:schechtman zinn}, the random vector $\frac{G}{\| G\|_2}$ is distributed on $\SSS^{n-1}$ according to $\sigma_{n-1}$. Thus,
\[
\sigma_{n-1}\left(u\in \SSS^{n-1}\,:\,\| x\|_2\,\left\| P_{ E_0} u\right\|_2\geq t\right)=\Pro\bigg(\|x\|_2\,\frac{\| P_{ E_0}G\|_2 }{\| G\|_2}\geq t\bigg)\,.
\]
Consequently, if $X$ is a random vector uniformly distributed on $\B_p^n$, $E\in \G_{n,k}$ is a random subspace independent of $X$ having distribution $\nu_{n,k}$, and $G$ is a standard Gaussian random vector in $\R^n$ that is independent of $X$ and $E$, we have that
\begin{align*}
&\Pro^{(X,E)}\left((x,F)\in\B_p^n\times\G_{n,k}:\| P_F x\|_2\geq t\right)\\
&=\frac{1}{\vert \B_p^n\vert}\int_{\B_p^n}\nu_{n,k}\left(E\in \G_{n,k}\,:\,\| P_E x\|_2\geq t\right)\dif x\cr
&=\frac{1}{\vert \B_p^n\vert}\int_{\B_p^n}\Pro\left(\| x\|_2\,\frac{\| P_{ E_0}G\|_2 }{\| G\|_2}\geq t\right)\dif x\cr
&=\Pro^{(X,G)}\left((x,g)\in\B_p^n\times\R^n:\| x\|_2\,\frac{\| P_{ E_0}g\|_2 }{\| g\|_2}\geq t\right)\,.
\end{align*}
Here, $\Pro^{(X,E)}$ denotes the joint distribution of the random vector $(X,E)\in\B_p^n\times\G_{n,k}$, while $\Pro^{(X,G)}$ stands for that of $(X,G)\in\B_p^n\times\R^n$. Now, let $Z=(Z_1,\ldots,Z_n)$ be a random vector having independent $p$-generalized Gaussian random entries. Then, by Proposition \ref{prop:schechtman zinn}, the random vector $U^{1/n}\frac{Z}{\Vert Z\Vert_p}$ is uniformly distributed in $\B_p^n$. Therefore,
\begin{align*}
&\Pro^{(X,G)}\left((x,g)\in\B_p^n\times\R^n:\| x\|_2\,\frac{\| P_{ E_0}g\|_2 }{\| g\|_2}\geq t\right)\\
&=\Pro^{(U,Z,G)}\left((u,z,g)\in[0,1]\times\R^n\times\R^n:u^{1/n}\frac{\| z\|_2}{\| z\|_p}\frac{\| P_{ E_0}g\|_2 }{\| g\|_2}\geq t\right)
\end{align*}
with $\Pro^{(U,Z,G)}$ being the joint distribution of the random vector $(U,Z,G)\in[0,1]\times\R^n\times\R^n$. Consequently, we conclude that the two random variables
$$
\| P_E X\|_2\qquad\text{and}\qquad
U^{1/n}\frac{\left(\sum_{i=1}^n Z_i^2\right)^{1/2}}{\left(\sum_{i=1}^n \vert Z_i\vert ^p \right)^{1/p}}\frac{\left( \sum_{i=1}^k g_i^2\right)^{1/2}}{\left( \sum_{i=1}^n g_i^2\right)^{1/2}}
$$
have the same distribution.
\end{proof}

\section{Proof of auxiliary LDPs}\label{sec:proof further LDPs}

The purpose of this section is to derive a number of auxiliary LDPs for the factors appearing in the probabilistic representation for $\|P_{E^{(n)}}X^{(n)}\|_2$ in Theorem \ref{RepresentationAnnealed}. These results can be seen as intermediate steps in the proof of Theorem \ref{thm:p>2}. Recall the set-up and the notation introduced above, define for each $n\in\N$ the random variables
\renewcommand\labelitemi{\tiny$\bullet$}
\begin{itemize}
\item $\displaystyle{V^{(n)}:=\frac{\Big(\sum_{i=1}^{k_n}g_i^2\Big)^{1/2}}{\Big(\sum_{i=1}^{n}g_i^2\Big)^{1/2}},}$
\item $\displaystyle{V_1^{(n)}:=U^{1/n}V^{(n)},}$
\item $\displaystyle{W^{(n)}:=n^{{1\over p}-{1\over 2}}\,\frac{\Big(\sum_{i=1}^{n}Z_i^2\Big)^{1/2}}{\Big(\sum_{i=1}^{n}|Z_i|^p\Big)^{1/p}}}$,
\end{itemize}
and the sequences $\bV:=(V^{(n)})_{n\in\N}$, $\bV_1:=(V_1^{(n)})_{n\in\N}$ and $\bW:=(W^{(n)})_{n\in\N}$. Using these definitions we notice that $n^{{1\over p}-{1\over 2}}\|P_{E^{(n)}}X^{(n)}\|_2$ has the same distribution as $V_1^{(n)}W^{(n)}$.

For technical reasons, we will have to split the LDPs for the sequences $\bV$, and $\bV_1$ into the three different cases
\renewcommand\labelitemi{\tiny$\bullet$}
\begin{itemize}
\item $\lambda\in(0,1)$,
\item $\lambda=0$,
\item $\lambda=1$,
\end{itemize}
where, $\lambda=\lim\limits_{n\to\infty}{k_n\over n}$. Note that the LDPs for the random sequences $\bV$ and $\bV_1$ will be unaffected by the choice of the value $p$. The latter enters only in the LDP for the random sequence $\bW$ and causes the different large deviation behavior of $\bPEX$ displayed Theorem \ref{thm:p>2} and Theorem \ref{thm:p<2}.

\subsection{LDP for the random sequence $\bV$: the case $\lambda\in(0,1)$}\label{subsec: LDP V_1^n}

The goal in this subsection is to prove an LDP for $\bV$ in the particular case that the parameter $\lambda$ satisfies $\lambda\in(0,1)$. To do this, we will make use of the following bound, which can be found in \cite[Lemma 5.1]{ACDP} and states that for any $k\geq 1$ and all $t\geq \max\{\sqrt{2(k-1)},1 \}$,
\begin{align}\label{eq:lemma Alonso et al}
t^{k-1}e^{-\frac{t^2}{2}} \leq \int_{t}^{\infty} r^k e^{-\frac{r^2}{2}}\,\dif r \leq 2 t^{k-1}e^{-\frac{t^2}{2}}\,.
\end{align}

\begin{lemma}\label{lem:GaussianPartLambdaIn(0,1)}
For each $n\in\N$ let $k_n\in\N$ with $k_n\in\{1,\ldots,n-1\}$ be a sequence such that
\[
\lim_{n\to\infty}\frac{k_n}{n} = \lambda\in(0,1).
\]
Then $\bV$ satisfies an LDP with speed $n$ and rate function
\[
\mathcal{I}_\bV^{(\lambda)}(y):=\begin{cases}
{\lambda\over 2}\log\big(\tfrac{\lambda}{y^2}\big)+{1-\lambda\over 2}\log\big(\tfrac{1-\lambda}{1-y^2}\big) &: y\in(0,1)\cr
+\infty &: \text{otherwise}\,.
\end{cases}
\]
\end{lemma}

\begin{proof}
Let us set, for each $n\in\N$,
\[
S_1^{(n)}:= \frac{1}{k_n}\sum_{i=1}^{k_n} (g_i^2,0)\qquad\text{and}\qquad
S_2^{(n)}:=\frac{1}{n-k_n}\sum_{i=k_n+1}^{n} (0,g_i^2)\,.
\]
First of all, note that, since $\lambda\notin\{0,1\}$, both $k_n$ and $n-k_n$ tend to $\infty$, as $n\to\infty$. By Cram\'er's theorem (see Proposition \ref{prop:cramer}), for any $A\in\mathscr L(\R^2)$, we have that
\begin{align*}
-\inf_{(x_1,x_2)\in A^\circ}\mathcal{I}_1^*(x_1,x_2) & \leq \liminf_{n\to\infty}\frac{1}{k_n}\log\big(\Pro\big(S_1^{(n)}\in A^\circ\big)\big) \cr
& \leq \limsup_{n\to\infty}\frac{1}{k_n}\log\big(\Pro\big(S_1^{(n)}\in \overline{A}\big)\big) \leq -\inf_{(x_1,x_2)\in \bar{A}}\mathcal{I}_1^*(x_1,x_2)\,.
\end{align*}
and
\begin{align*}
-\inf_{(x_1,x_2)\in A^\circ}\mathcal{I}_2^*(x_1,x_2) & \leq \liminf_{n\to\infty}\frac{1}{n-k_n}\log\big(\Pro\big(S_2^{(n)}\in A^\circ\big)\big) \cr
& \leq \limsup_{n\to\infty}\frac{1}{n-k_n}\log\big(\Pro\big(S_2^{(n)}\in \overline{A}\big)\big) \leq -\inf_{(x_1,x_2)\in \bar{A}}\mathcal{I}_2^*(x_1,x_2)\,,
\end{align*}
where $\mathcal{I}^*_1$ is the Legendre-Fenchel transform of the function
\begin{align*}
& \mathcal{I}_1(t_1,t_2) = \log(\E e^{\left\langle (t_1,t_2), (g_1^2,0)\right\rangle}) \cr
& = \log\Big( \int_{\R} \int_{\R} e^{t_1x_1^2}\frac{1}{\sqrt{2\pi}}e^{-x_1^2/2}\frac{1}{\sqrt{2\pi}}e^{-x_2^2/2}\,\dif x_1\dif x_2\Big)\,,\qquad (t_1,t_2)\in\R^2\,,
\end{align*}
and $\mathcal{I}^*_2$ is the Legendre-Fenchel transform of
\begin{align*}
& \mathcal{I}_2(t_1,t_2)  = \log(\E e^{\left\langle (t_1,t_2), (0,g_1^2)\right\rangle}) \cr
&= \log\Big( \int_{\R} \int_{\R} e^{t_2x_2^2}\frac{1}{\sqrt{2\pi}}e^{-x_1^2/2}\frac{1}{\sqrt{2\pi}}e^{-x_2^2/2}\,\dif x_1\dif x_2\Big)\,,\qquad (t_1,t_2)\in\R^2\,.
\end{align*}
Note that in both cases, we obviously have that the point $(0,0)\in\R^2$ belongs to the effective domains of $\mathcal I_1$ and of $\mathcal I_2$. We now compute the rate functions explicitly. For that purpose, let, for any $t\in\R$,
\begin{align*}
\mathcal{I}(t) &:=  \log\big(\E e^{tg_1^2}\big)\cr
& = \log\left( \int_{\R}e^{tx^2}\frac{1}{\sqrt{2\pi}}e^{-x^2/2}\,\dif x\right) \cr
& = \log\left( 2\int_{0}^\infty\frac{1}{\sqrt{2\pi}}e^{-(1-2t)x^2/2}\,\dif x\right) \cr
& =
\begin{cases}
 -\frac{1}{2}\log(1-2t) &: t<\frac{1}{2}\cr
 +\infty &: \text{otherwise}\,.
\end{cases}
\end{align*}
Then the Legendre-Fenchel transform of $\mathcal{I}$ is given by
\[
\mathcal{I}^*(x) = \sup_{t\in\R}\big[ xt-\mathcal{I}(t)\big] = \sup_{t<\frac{1}{2}}\Big[ xt+\frac{1}{2}\log(1-2t)\Big]\,, \qquad x\in\R\,.
\]
If $x>0$, then the supremum is attained at $t_0:=\frac{x-1}{2x}<\frac{1}{2}$
and so
\[
\mathcal{I}^*(x) = \frac{x-1}{2}-\frac{1}{2}\log(x)\,,
\]
If $x\leq 0$, then the function $f:(-\infty,\tfrac{1}{2})\to\R$, $ f(t)=xt+\frac{1}{2}\log(1-2t)$ is non-increasing, so
\[
\sup_{t<\frac{1}{2}}f(t) = \lim_{t\to-\infty}f(t)=+\infty\,.
\]
Thus,
\[
\mathcal{I}^*(x)=
\begin{cases}
\frac{x-1}{2}-\frac{1}{2}\log(x) &: x>0\cr
+\infty &: x\leq 0\,.
\end{cases}
\]
Note that, for all $(t_1,t_2)\in\R^2$,
\[
\mathcal{I}_1(t_1,t_2) = \mathcal{I}(t_1)\qquad\text{and}\qquad \mathcal{I}_2(t_1,t_2)= \mathcal{I}(t_2)\,.
\]
Thus, for $x=(x_1,x_2)\in\R^2$,
\begin{align*}
\mathcal{I}_1^*(x_1,x_2) & = \sup_{t=(t_1,t_2)\in\R^2}\left[\langle t, x \rangle - \mathcal{I}_1(t_1,t_2)\right] \cr
& = \sup_{t_1,t_2\in\R} \left[t_1x_1 + t_2x_2 - \mathcal{I}(t_1)\right] \cr
& = \mathcal{I}^*(x_1)+\sup_{t_2\in\R}[x_2t_2] \cr
& =
\begin{cases}
\frac{x_1-1}{2}-\frac{1}{2}\log(x_1) &: x_2=0\text{ and }x_1>0\cr
+\infty &: \text{otherwise}\,.
\end{cases}
\end{align*}
Similarly, for $x=(x_1,x_2)\in\R^2$, we obtain
\begin{align*}
\mathcal{I}_2^*(x_1,x_2)
& =
\begin{cases}
\frac{x_2-1}{2}-\frac{1}{2}\log(x_2) &: x_1=0\text{ and }x_2>0\cr
+\infty &: \text{otherwise}\,.
\end{cases}
\end{align*}
Note that the sequence $(S_1^{(n)})_{n\in\N}$ satisfies an LDP with speed $n$ and rate function $\lambda \mathcal{I}_1^*$ and that  $(S_2^{(n)})_{n\in\N}$ satisfies an LDP with speed $n$ and rate function $(1-\lambda) \mathcal{I}_2^*$.
For $n\in\N$, let
$$S^{(n)}:=\frac{k_n}{n}S_1^{(n)}+\frac{n-k_n}{n}S_2^{(n)}=\frac{1}{n}\left(\sum_{i=1}^{k_n} g_i^2,\sum_{i=k_n+1}^{n} g_i^2\right)\,.
$$
Define $\lambda_n=\frac{k_n}{n}$ and let $F_n\,:\,\R^4\to\R^2$ be the function given by
$$
F_n(x_1,x_2,y_1,y_2)=\lambda_n(x_1,x_2)+(1-\lambda_n)(y_1,y_2)\,.
$$
Then notice that, for each $n\in\N$,
$$
S^{(n)}=F_n(S_1^{(n)},S_2^{(n)})\,.
$$

Let $F\,:\,\R^4\to\R^2$ be the function given by
$$
F(x_1,x_2,y_1,y_2)=\lambda(x_1,x_2)+(1-\lambda)(y_1,y_2)
$$
and denote by $d(\,\cdot\,,\,\cdot\,)$ the distance in $\R^2$ given by the norm $\Vert\,\cdot\,\Vert_1$, i.e.,
$$d\big((x_1,x_2),(y_1,y_2)\big)=\vert x_1-y_1\vert+\vert x_2-y_2\vert\,.
$$
Then,
\begin{align*}
d\left(F_n(S_1^{(n)},S_2^{(n)}),F(S_1^{(n)},S_2^{(n)})\right)&=\vert\lambda_n-\lambda\vert\left(\frac{1}{k_n}\sum_{i=1}^{k_n} g_i^2+\frac{1}{n-k_n}\sum_{i=k_n+1}^{n} g_i^2\right)\cr
&\leq\frac{\vert\lambda_n-\lambda\vert}{\min\{k_n,n-k_n\}}\sum_{i=1}^{n} g_i^2.
\end{align*}
Hence, for any $\delta>0$,
\begin{align*}
\Pro\left(d\big(F_n(S_1^{(n)},S_2^{(n)}),F(S_1^{(n)},S_2^{(n)})\big)>\delta\right)&\leq\Pro\left(\| G\|_2^2>\frac{\delta\min\{k_n,n-k_n\}}{\vert\lambda_n-\lambda\vert}\right)\cr
&=\Pro\left(\| G\|_2>\sqrt{\frac{\delta\min\{k_n,n-k_n\}}{\vert\lambda_n-\lambda\vert}}\right)\,,
\end{align*}
where $G$ is a standard Gaussian random vector in $\R^n$. Note that if we define
 $$
\alpha_n:= \sqrt{\frac{\min\{k_n,n-k_n\}}{\vert\lambda_n-\lambda\vert}}= \sqrt n\,\sqrt{\frac{\min\{\lambda_n,1-\lambda_n\}}{\vert\lambda_n-\lambda\vert}},
 $$
we have that $\frac{\alpha_n^2}{n}\to\infty$, as $n\to\infty$, since $\lambda_n\to\lambda\in(0,1)$, as $n\to\infty$, and $\lambda\notin\{0,1\}$. By integration in spherical coordinates and Stirling's formula the latter probability equals
 \begin{align}\label{eq:integral estiamte}
 n\vert \B_2^n\vert\int_{\sqrt\delta\alpha_n}^\infty r^{n-1}\frac{e^{-r^2/2}}{(\sqrt{2\pi})^n}\,\dif r&= \frac{n}{2^{n/2}\Gamma\left(1+\frac{n}{2}\right)}\int_{\sqrt\delta\alpha_n}^\infty r^{n-1}e^{-r^2/2}\,\dif r\cr
 &\leq \frac{cn}{n^{n/2}e^{-n/2}\sqrt{\pi n}}\int_{\sqrt\delta\alpha_n}^\infty r^{n-1}e^{-r^2/2}\,\dif r\,,
 \end{align}
 where $c\in(0,\infty)$ is an absolute constant. Now, we estimate the last integral from above using \eqref{eq:lemma Alonso et al}. Since $\frac{\alpha_n^2}{n}\to\infty$, as $n\to\infty$, we have that, for any $\delta>0$, $\sqrt\delta\alpha_n$ is much greater than $\sqrt n$ whenever $n$ is sufficiently large. This means that we can apply \eqref{eq:lemma Alonso et al} with the choice $k=n-1$ and $t=\sqrt{\delta}\alpha_n$ there, and deduce that for every $\delta>0$ there exists $n_0\in\N$ such that if $n\geq n_0$, the last expression in \eqref{eq:integral estiamte} is bounded from above by
 \[
  \frac{2cn}{n^{n/2}e^{-n/2}\sqrt{\pi n}} \big(\sqrt\delta\alpha_n\big)^{n-2}e^{-\delta\alpha_n^2/2}\,.
 \]
 Consequently, for any $\delta>0$ there exists $n_0\in\N$ such that if $n\geq n_0$,
  \begin{align*}
  \frac{1}{n}&\log\Pro\left(d\big(F_n(S_1^{(n)},S_2^{(n)}),F(S_1^{(n)},S_2^{(n)})\big)>\delta\right) \leq \frac{1}{n}\log\left(\frac{2c}{\sqrt\pi}\right)-\frac{1}{2n}\log n +\frac{1}{2}\cr
  &+\frac{n-2}{2n}\log\left(\frac{\delta\min\{\lambda_n,1-\lambda_n\}}{2\vert\lambda_n-\lambda\vert}\right)-\frac{\delta\min\{\lambda_n,1-\lambda_n\}}{2\vert\lambda_n-\lambda\vert}\,,
   \end{align*}
 which tends to $-\infty$, as $n\to\infty$, since $\lambda_n\to\lambda\notin\{0,1\}$. This means that we can apply Proposition \ref{prop:refinement contraction principle} to deduce that sequence of random vectors
\[
(S^{(n)})_{n\in\N}=\left(F_n(S_1^{(n)},S_2^{(n)})\right)_{n\in\N}=\left(\frac{1}{n}\left(\sum_{i=1}^{k_n} g_i^2,\sum_{i=k_n+1}^{n} g_i^2\right)\right)_{n\in\N}
\]
satisfies an LDP with speed $n$ and rate function given, for any $x=(x_1,x_2)\in\R^2$, by
\begin{align*}
\mathcal{J}(x_1,x_2)& = \inf_{y=(y_1,y_2), z=(z_1,z_2)\in\R^2\atop \lambda y+(1-\lambda)z=x} \left[\lambda \mathcal{I}_1^*(y_1,y_2) + (1-\lambda)\mathcal{I}^*_2(z_1,z_2)\right] \cr
& = \inf_{y=(y_1,0), z=(0,z_2)\in\R^2\atop  \lambda y+(1-\lambda)z=x} [\lambda \mathcal{I}_1^*(y_1,0) + (1-\lambda)\mathcal{I}^*_2(0,z_2)] \cr
& = \lambda \mathcal{I}_1^*\Big(\frac{x_1}{\lambda},0\Big) + (1-\lambda)\mathcal{I}^*_2\Big(0,\frac{x_2}{1-\lambda}\Big)\cr
&= \lambda \mathcal{I}^*\Big(\frac{x_1}{\lambda}\Big) + (1-\lambda)\mathcal{I}^*\Big(\frac{x_2}{1-\lambda}\Big)\,.
\end{align*}
Now, notice that for each $n\in\N$, $V^{(n)}=G(S^{(n)})$, with the function $G\,:\,\R^2\to\R$ given by
$$
G(x_1,x_2)=\frac{(x_1)^{1/2}}{(x_1+x_2)^{1/2}}\,.
$$
Thus, by the contraction principle (see Proposition \ref{prop:contraction principle}),  the random sequence $\bV$ satisfies an LDP with speed $n$ and rate function
\begin{align*}
\mathcal{I}_\bV^{(\lambda)}(y) & = \inf_{x_1,x_2\in\R\atop \frac{(x_1)^{1/2}}{(x_1+x_2)^{1/2}}=y}\left[\lambda \mathcal{I}^*\left(\frac{x_1}{\lambda}\right) + (1-\lambda)\mathcal{I}^*\left(\frac{x_2}{1-\lambda}\right)\right]\cr
& = \inf_{x_1,x_2>0\atop \frac{(x_1)^{1/2}}{(x_1+x_2)^{1/2}}=y}\left( \lambda \left[\frac{x_1/\lambda-1}{2}-\frac{1}{2}\log\Big({x_1\over\lambda}\Big)\right] \right.\cr
& \qquad\qquad\qquad\qquad \left.+ (1-\lambda)\left[\frac{x_2/(1-\lambda)-1}{2}-\frac{1}{2}\log\Big({x_2\over 1-\lambda}\Big)\right]\right)
\end{align*}
for any $y\in(0,1)$ and $+\infty$ otherwise. This is so, because if $y=0$, then $x_1=0$ and $\mathcal{I}^*(0)=+\infty$. If $y=1$, then $x_2=0$ and $\mathcal{I}^*(0)=+\infty$. If $y\notin[0,1]$, $G^{-1}(\{y\})=\emptyset$ and we have that $\mathcal{I}_\bV(y)=+\infty$.
Let us compute this infimum. Note that, if $y\in(0,1)$,
\[
 \frac{(x_1)^{1/2}}{(x_1+x_2)^{1/2}}=y \quad\text{if and only if}\quad x_2=\frac{x_1(1-y^2)}{y^2}\,,
\]
whence
\begin{align*}
\mathcal{I}_\bV^{(\lambda)}(y) & = \inf_{x_1>0}\left[\frac{1}{2}\frac{ x_1}{y^2} - \frac{1}{2}-\frac{\lambda}{2}\log\left(\frac{ x_1}{\lambda}\right)-\frac{1-\lambda}{2}\log\left(\frac{x_1(1-y^2)}{(1-\lambda)y^2}\right) \right] \cr
& = \inf_{x_1>0} \left[\frac{1}{2}\frac{x_1}{y^2} - \frac{1}{2}\log(x_1) - \frac{1}{2}-\frac{1-\lambda}{2}\log\left(\frac{\lambda (1-y^2)}{(1-\lambda)y^2}\right)+ \frac{1}{2}\log(\lambda)\right]\,.
\end{align*}
This infimum is attained at $x_1=y^2$, which implies that
\[
\mathcal{I}_\bV^{(\lambda)}(y) = \frac{\lambda}{2}\log\left(\frac{\lambda}{y^2}\right) +\frac{1-\lambda}{2}\log\left(\frac{1-\lambda}{1-y^2}\right),
\]
whenever $y\in(0,1)$.
\end{proof}

\subsection{LDP for the random sequence $\bV$: the case $\lambda=0$}\label{subsec: LDP V_1^n lambda=0}

We will make use of the following slice integration formula on the sphere, which can be found, for instance, in \cite[Theorem A.4]{AxlerEtAl}. For a non-negative measurable function $f:\mathbb{S}^{n-1}\to\mathbb{R}$ and a fixed $k\in\{1,\ldots,n-1\}$ it says that
\begin{equation}\label{eq:SliceSphere}
\begin{split}
&\int_{\mathbb{S}^{n-1}}f(x)\,\dif\sigma_{n-1}(x) \\
&\qquad = {k\over n}{|\B_2^{k}|\over|\B_2^n|}\int_{\B_2^{n-k}}\left(1-\|x\|_2^2\right)^{k-2\over 2}\int_{\mathbb{S}^{k-1}}f\Big(x,\sqrt{1-\|x\|_2^2}\,y\Big)\,\dif\sigma_{k-1}(y)\dint x\,.
\end{split}
\end{equation}
We can now prove the LDP for $\bV$ under the assumption that $\lambda=0$.

\begin{lemma}\label{lem: lamda zero without U}
For each $n\in\N$ let $k_n\in\N$ with $k_n\in\{1,\ldots,n-1\}$ be a sequence such that
\[
\lim_{n\to\infty}\frac{k_n}{n} = 0.
\]
Then the sequence $\bV$ of random variables satisfies an LDP with speed $n$ and rate function
\[
\mathcal{I}_\bV^{(0)}(y):=\begin{cases}
-\frac{1}{2}\log(1-y^2) &: y\in[0,1)\\
+\infty &: \text{otherwise}\,.
\end{cases}
\]
\end{lemma}
\begin{proof}
For each $n\in\N$, let $\widetilde{X}^{(n)}$ be a random vector that is uniformly distributed on the sphere $\SSS^{n-1}$. Then, for each $n\in\N$, $P_{E^{(n)}}\widetilde{X}^{(n)}$ has the same distribution as $V^{(n)}$. By the slice integration formula \eqref{eq:SliceSphere}, letting $E_0^{(n)}:=\textrm{span}(\{e_1,\dots,e_{k_n}\})$, for any $0<\alpha_1<\alpha_2\leq 1$ and all $n\in\N$, we have that
\begin{align}
\nonumber&\Pro\left(\big\|P_{E_0^{(n)}}\widetilde{X}^{(n)}\big\|_2\in[\alpha_1,\alpha_2]\right)\\
\nonumber&\qquad={(n-k_n)|\B_2^{n-k_n}|\over n|\B_2^n|}\int_{\{x\in \B_2^{k_n}:\|x\|_2\in[\alpha_1,\alpha_2]\}}(1-\|x\|_2^2)^{{n-k_n\over 2}-1}\,\dint x\\
\label{eq:Estimate1xxx}&\qquad={(n-k_n)k_n|\B_2^{n-k_n}||\B_2^{k_n}|\over n|\B_2^n|}\int_{\alpha_1}^{\alpha_2}r^{k_n-1}(1-r^2)^{{n-k_n\over 2}-1}\,\dint r\,.
\end{align}
Assuming $k_n\geq 2$ we bound this integral from above as follows:
\begin{align*}
\Pro\left(\big\|P_{E_0^{(n)}}\widetilde{X}^{(n)}\big\|_2\in[\alpha_1,\alpha_2]\right) &\leq \Big(1-{k_n\over n}\Big)k_n\,c_{n,k_n}\int_{\alpha_1}^1 r(1-r^2)^{{n-k_n\over 2}-1}\,\dint r\\
&=\Big(1-{k_n\over n}\Big)k_n{c_{n,k_n}\over 2}{(1-\alpha_1^2)^{n-k_n\over 2}\over{n-k_n\over 2}}\,,
\end{align*}
where, in view of \eqref{eq:VolBall},
\[
c_{n,k_n}:=\frac{|\B_2^{n-k_n}||\B_2^{k_n}|}{|\B^n_2|}=\frac{\Gamma\big(1+\frac{n}{2}\big)}{\Gamma\left(1+\frac{n-k_n}{2}\right)\Gamma\big(1+\frac{k_n}{2}\big)}\,.
\]
Whence, using Stirling's formula and taking into account that $\frac{k_n}{n}\to 0$ as $n\to\infty$, we conclude that
\begin{align*}
&\limsup_{n\to\infty}{1\over n}\log\Pro\left(\big\|P_{E_0^{(n)}}\widetilde{X}^{(n)}\big\|_2\in[\alpha_1,\alpha_2]\right) \\
&\leq \limsup_{n\to\infty}\bigg[{\log(1-{k_n\over n})\over n}+{\log (k_n)\over n}+{\log (c_{n,k_n})\over n}\\
&\qquad\qquad\qquad\qquad+{n-k_n\over 2n}\log(1-\alpha_1^2)-{\log(n-k_n)\over n}\bigg]\\
&={1\over 2}\log(1-\alpha_1^2)\\
&=-\inf_{y\in[\alpha_1,\alpha_2]}\mathcal{I}_\bV^{(0)}(y)\,.
\end{align*}
Using once more \eqref{eq:Estimate1xxx}, for the lower bound we compute
\begin{align*}
&\Pro\big(\big\|P_{E_0^{(n)}}\widetilde{X}^{(n)}\big\|_2\in(\alpha_1,\alpha_2)\big) \\
&\geq \Big(1-{k_n\over n}\Big)k_n\,c_{n,k_n}\,\alpha_1^{k_n-2}\int_{\alpha_1}^{\alpha_2}r(1-r^2)^{{n-k_n\over 2}-1}\,\dint r\\
& = \Big(1-{k_n\over n}\Big)k_nc_{n,k_n}\alpha_1^{k_n-2}{(1-\alpha_1^2)^{n-k_n\over 2}\over n-k_n}\left(1-\Big({1-\alpha_2^2\over 1-\alpha_1^2}\Big)^{n-k_n\over 2}\right).
\end{align*}
Thus, since $\frac{k_n}{n}\to 0$ for $n\to\infty$ and because $\log(1-x)$ behaves like $-x$ for small arguments $x$,
\begin{align*}
&\liminf_{n\to\infty}{1\over n}\log\Pro\big(\big\|P_{E_0^{(n)}}\widetilde{X}^{(n)}\big\|_2\in(\alpha_1,\alpha_2)\big)\cr
& \geq \liminf_{n\to\infty}\Big[{\log(1-{k_n\over n})\over n}+{\log (k_n)\over n}+{\log (c_{n,k_n})\over n}+{k_n-2\over n}\log(\alpha_1) \cr
&\qquad\qquad+{n-k_n\over 2n}\log(1-\alpha_1^2)-{\log(n-k_n)\over n}+{1\over n}{\log\Big(1-\Big({1-\alpha_2^2\over 1-\alpha_1^2}\Big)^{n-k_n\over 2}\Big)}\Big]\\
&={1\over 2}\log(1-\alpha_1^2)\\
&=-\inf_{y\in(\alpha_1,\alpha_2)}\mathcal{I}_\bV^{(0)}(y)\,.
\end{align*}
If $k_n=1$, bounding $r^{k_n-2}$ from above by $\alpha_1^{-1}$ and from below by $1$, we obtain the same estimates. If $\alpha_1=0$ and $\alpha_2\in(0,1]$, the relation
\[
\limsup_{n\to\infty}\frac{\log\big(\Pro\big(\big\|P_{E_0^{(n)}}\widetilde{X}^{(n)}\big\|_2\in[0,\alpha_2]\big)\big)}{n} \leq 0
\]
is trivial.
Since $-\inf_{y\in[0,\alpha_2]}\mathcal{I}_\bV^{(0)}(y)=0$, we obtain the upper bound. For the lower bound we notice that for any $\varepsilon>0$,
\begin{align*}
&\liminf_{n\to\infty}{1\over n}\log\Pro\big(\big\|P_{E_0^{(n)}}\widetilde{X}^{(n)}\big\|_2\in(0,\alpha_2)\big)\\
& \geq\liminf_{n\to\infty}{1\over n}\log\Pro\big(\big\|P_{E_0^{(n)}}\widetilde{X}^{(n)}\big\|_2\in(\varepsilon,\alpha_2)\big)
\cr
&={1\over 2}\log(1-\varepsilon^2)\,.
\end{align*}
Letting $\varepsilon\to 0^+$, because of $-\inf_{y\in(0,\alpha_2)}\mathcal{I}_\bV^{(0)}(y)=0$, we obtain the lower bound as well. Since $\mathcal{I}_\bV^{(0)}(y)=+\infty$ for $y\notin[0,1)$ the inequalities also hold for intervals $(\alpha_1,\alpha_2)$ and $[\alpha_1,\alpha_2]$ if $\alpha_1<0$ or $\alpha_2> 1$.

Now, the family of open intervals is a basis for the standard topology on $\R$. Hence, by Proposition \ref{prop:basis topology}, the sequence $\bV$ satisfies a weak LDP with speed $n$ and rate function $\mathcal{I}_\bV^{(0)}$. Since, for every $n\in\N$, the random variable $V^{(n)}$ only takes values in $[0,1]$, we have that for every closed set $A\subset\R$ the compact set $A\cap [0,1]\subset\R$ has the same probability. Besides,
$$
\inf_{y\in A}\mathcal{I}_\bV^{(0)}(y)=\inf_{y\in A\cap [0,1]}\mathcal{I}_\bV^{(0)}(y)\,.
$$
Consequently, by Proposition \ref{prop:equivalence weak and full LDP}, $\bV$ satisfies a full LDP with speed $n$ and rate function $\mathcal{I}_\bV^{(0)}$.
\end{proof}

\subsection{LDP for the random sequence $\bV$: the case $\lambda=1$}\label{subsec: LDP V_1^n lambda=1}

Finally, we consider the set-up in which $\lambda=1$. By transition to orthogonal complement subspaces, we reduce this situation to the previously considered case $\lambda=0$ and conclude the result from Lemma \ref{lem: lamda zero without U}.

\begin{lemma}\label{lem: lamda one without U}
For each $n\in\N$ let $k_n\in\N$ with $k_n\in\{1,\ldots,n-1\}$ be a sequence such that
\[
\lim_{n\to\infty}\frac{k_n}{n} = 1\,.
\]
Then the sequence of random variables $\bV$ satisfies an LDP with speed $n$ and rate function
\[
\mathcal{I}_\bV^{(1)}(y):=\begin{cases}
-\frac{1}{2}\log(y^2) &: y\in(0,1]\cr
+\infty &: \text{otherwise}\,.
\end{cases}
\]
\end{lemma}
\begin{proof}
Defining $E_0^{(n)}:=\textrm{span}(\{e_1,\ldots,e_{k_n}\})$, we notice that $\|P_{E_0^{(n)}}\widetilde{X}\|_2$ has the same distribution as $V^{(n)}$, where $\widetilde{X}$ is a uniform random point on $\mathbb{S}^{n-1}$. Let $(\alpha_1,\alpha_2)$ be an open interval with $0\leq \alpha_1\leq \alpha_2<1$. Then,
\begin{align*}
\Pro\big(V^{(n)}\in(\alpha_1,\alpha_2)\big) & = \Pro\left(\big\|P_{E_0^{(n)}}\widetilde{X}^{(n)}\big\|_2\in(\alpha_1,\alpha_2)\right)\cr
& =\Pro\left(\alpha_1^2<1-\big\|P_{(E_0^{(n)})^\perp}X\big\|_2^2<\alpha_2^2\right)\\
&=\Pro\left(\big\|P_{\overline{E}_0^{(n)}}\widetilde{X}^{(n)}\big\|_2\in\Big(\sqrt{1-\alpha_2^2},\sqrt{1-\alpha_1^2}\,\Big)\right)
\end{align*}
with $\overline{E}_0^{(n)}:={\rm span}(\{e_1,\ldots,e_{n-k_{n}}\})$. The same holds if the open interval $(\alpha_1,\alpha_2)$ is replaced by the closed interval $[\alpha_1,\alpha_2]$. Consequently, if $\mathcal{I}_\bV^{(0)}$ denotes the rate function from Lemma \ref{lem: lamda zero without U} and $\mathcal{I}_\bV^{(1)}$ denotes the function defined in the statement of the present lemma, we find that, since $\frac{n-k_n}{n}\to0$, as $n\to\infty$,
\begin{align*}
&\liminf_{n\to\infty}{1\over n}\log\Pro\big(V^{(n)}\in(\alpha_1,\alpha_2)\big)\\
&= \liminf_{n\to\infty}{1\over n}\log\Pro\left(\big\|P_{\overline{E}_0^{(n)}}\widetilde{X}^{(n)}\big\|_2\in\Big(\sqrt{1-\alpha_2^2},\sqrt{1-\alpha_1^2}\,\Big)\right)\\
&\geq -\inf_{y\in (\sqrt{1-\alpha_2^2},\sqrt{1-\alpha_1^2})}\mathcal{I}_\bV^{(0)}(y)\\
&={1\over 2}\log \alpha_2^2\\ &=-\inf_{y\in(\alpha_1,\alpha_2)}\mathcal{I}_\bV^{(1)}(y)\,.
\end{align*}
Similarly, we get
\begin{align*}
&\limsup_{n\to\infty}{1\over n}\log\Pro\left(V^{(n)}\in[\alpha_1,\alpha_2]\right)\\
&= \limsup_{n\to\infty}{1\over n}\log\Pro\left(\big\|P_{\overline{E}_0^{(n)}}\widetilde{X}^{(n)}\big\|_2\in\Big[\sqrt{1-\alpha_2^2},\sqrt{1-\alpha_1^2}\,\Big]\right)\\
&\leq -\inf_{y\in [\sqrt{1-\alpha_2^2},\sqrt{1-\alpha_1^2}]}\mathcal{I}_\bV^{(0)}(y)\\
&={1\over 2}\log \alpha_2^2\\
&=-\inf_{y\in[\alpha_1,\alpha_2]}\mathcal{I}_\bV^{(1)}(y)\,.
\end{align*}
Since $\mathcal{I}_\bV^{(1)}(y)=+\infty$ for $y\notin(0,1]$ the inequalities also hold for intervals $(\alpha_1,\alpha_2)$ if $\alpha_1<0$ or $\alpha_2> 1$. Literally the same argument already used at the end of the proof of Lemma \ref{lem: lamda zero without U} completes the proof.
\end{proof}


\subsection{LDP for the random sequence $\bV_1$}\label{subsec:LDPV1}

In this subsection we will prove LDPs for the sequence $\bV_1$, again in the three different cases $\lambda\in(0,1)$, $\lambda=1$ and $\lambda=0$. As we will see below, the radial part is in fact negligible and so the rate functions for $\bV_1$ coincide with the corresponding ones for $\bV$ obtained in Subsections \ref{subsec: LDP V_1^n}, \ref{subsec: LDP V_1^n lambda=0} and \ref{subsec: LDP V_1^n lambda=1}.

In all three cases, we will use the following result proved in \cite[Lemma 3.3]{GKR}.

\begin{lemma}\label{lem:gantert rate function U}
The sequence $\bU=(U^{1/n})_{n\in\N}$ satisfies an LDP with speed $n$ and rate function
\[
\mathcal{I}_\bU(y):=
\begin{cases}
-\log(y) &: y\in(0,1]\\
+\infty &: \text{otherwise}\,.
\end{cases}
\]
\end{lemma}

We start with $\lambda\in(0,1)$, in which case the LDP for $\bV_1$ is a consequence of Lemma \ref{lem:GaussianPartLambdaIn(0,1)}.

\begin{cor}\label{cor:lambda 0,1 including U}
For each $n\in\N$ let $k_n\in\N$ with $k_n\in\{1,\ldots,n-1\}$ and assume that
\[
\lim_{n\to\infty}\frac{k_n}{n} = \lambda\in(0,1)\,.
\]
Then the sequence $\bV_1$ satisfies an LDP with speed $n$ and rate function
\[
\mathcal{I}_{\bV_1}^{(\lambda)}(y):=\begin{cases}
{\lambda\over 2}\log\big(\tfrac{\lambda}{y^2}\big)+{1-\lambda\over 2}\log\big(\tfrac{1-\lambda}{1-y^2}\big) &: y\in(0,1)\cr
+\infty &: \text{otherwise}\,.
\end{cases}
\]
\end{cor}
\begin{proof}
By Proposition \ref{JointRateFunction}, Lemma \ref{lem:GaussianPartLambdaIn(0,1)} and Lemma \ref{lem:gantert rate function U} the sequence of random vectors $\big((U^{1/n},V^{(n)})\big)_{n\in\N}$ satisfies an LDP with speed $n$ and rate function
\[
\mathcal{I}(x_1,x_2)=\begin{cases}
-\log(x_1)+{\lambda\over 2}\log\big({\lambda\over x_2^2}\big)+{1-\lambda\over 2}\log\big({1-\lambda\over 1-x_2^2 }\big) &: x_1\in(0,1]\text{ and }x_2\in(0,1) \\
+\infty &: \text{otherwise}\,,
\end{cases}
\]
$(x_1,x_2)\in\R^2$. Defining $F(x_1,x_2):=x_1x_2$ and applying the contraction principle (see Proposition \ref{prop:contraction principle}), we deduce that $\bV_1=\big(F(U^{1/n},V^{(n)})\big)_{n\in\N}$ satisfies an LDP with speed $n$ and rate function
\[
\mathcal{I}_{\bV_1}^{(\lambda)}(y) = \inf_{(x_1,x_2)\in\R^2\atop F(x_1,x_2)=y}\mathcal{I}(x_1,x_2)\,,\qquad y\in\R\,.
\]
If $y\in(0,1)$, then
\begin{align*}
&\mathcal{I}_{\bV_1}^{(\lambda)}(y) \\
&= \inf_{x_1x_2=y}\Big[-\log (x_1)+{\lambda\over 2}\log\Big({\lambda\over x_2^2}\Big)+{1-\lambda\over 2}\log\Big({1-\lambda\over 1-x_2^2 }\Big)\Big]\\
&=\inf_{x_1x_2=y}\Big[-{\lambda\over 2}\log (x_1^2)-{1-\lambda\over 2}\log (x_1^2)+{\lambda\over 2}\log\Big({\lambda\over x_2^2}\Big)+{1-\lambda\over 2}\log\Big({1-\lambda\over 1-x_2^2 }\Big)\Big]\\
&=\inf_{x_1x_2=y}\Big[{\lambda\over 2}\log\Big({\lambda\over (x_1x_2)^2}\Big)+{1-\lambda\over 2}\log\Big({1-\lambda\over x_1^2-(x_1x_2)^2}\Big)\Big]\\
&=\inf_{x_1x_2=y}\Big[{\lambda\over 2}\log\Big({\lambda\over y^2}\Big)+{1-\lambda\over 2}\log\Big({1-\lambda\over x_1^2-y^2}\Big)\Big]\\
&={\lambda\over 2}\log\Big({\lambda\over y^2}\Big)+{1-\lambda\over 2}\log\Big({1-\lambda\over 1-y^2}\Big)\,,
\end{align*}
since the infimum is attained at $x_1=1$ and $x_2=y$. If $y\notin(0,1)$, for every $(x_1,x_2)\in\R^2$ such that $x_1x_2=y$, we have $\mathcal{I}(x_1,x_2)=+\infty$.
\end{proof}

In the same way we obtained Corollary \ref{cor:lambda 0,1 including U}, we also treat the LDP for $\bV_1$ if $\lambda=0$. In this situation the result is a consequence of Lemma \ref{lem: lamda zero without U}.

\begin{cor}\label{cor: rate function V_1 lambda=0}
For each $n\in\N$ let $k_n\in\N$ with $k_n\in\{1,\ldots,n-1\}$ be such that
\[
\lim_{n\to\infty}\frac{k_n}{n} = 0\,.
\]
Then, the sequence of random variables $\bV_1$ satisfies an LDP with speed $n$ and rate function
\[
\mathcal{I}_{\bV_1}^{(0)}(y):=\begin{cases}
-\frac{1}{2}\log(1-y^2) &: y\in[0,1)\cr
+\infty &: \text{otherwise}\,.
\end{cases}
\]
\end{cor}
\begin{proof}
By Proposition \ref{JointRateFunction}, Lemma \ref{lem: lamda zero without U} and Lemma \ref{lem:gantert rate function U} the sequence of random vectors $\big((U^{1/n},V^{(n)})\big)_{n\in\N}$ satisfies an LDP with speed $n$ and rate function
\[
\mathcal{I}(x_1,x_2)=\begin{cases}
-\log (x_1) - \frac{1}{2}\log(1-x_2^2) &: x_1\in(0,1]\text{ and }x_2\in[0,1) \\
+\infty &: \text{otherwise}\,,
\end{cases}
\]
$(x_1,x_2)\in\R^2$.
Defining $F(x_1,x_2):=x_1x_2$ and applying the contraction principle (see Proposition \ref{prop:contraction principle}), we deduce that $\bV_1=\big(F(U^{1/n},V^{(n)})\big)_{n\in\N}$ satisfies an LDP with speed $n$ and rate function
\[
\mathcal{I}_{\bV_1}^{(0)}(y) = \inf_{(x_1,x_2)\in\R^2\atop F(x_1,x_2)=y}\mathcal{I}(x_1,x_2)\,,\qquad y\in\R\,.
\]
If $y\in[0,1)$, then
\begin{align*}
\mathcal{I}_{\bV_1}^{(0)}(y) &= \inf_{x_1x_2=y}\Big[-\log (x_1) -\frac{1}{2}\log(1-x_2^2)\Big]\\
&=\inf_{x_1x_2=y}\Big[-\frac{1}{2}\log(x_1^2)-\frac{1}{2}\log(1-x_2^2)\Big]\\
&=\inf_{x_1x_2=y}\Big[-\frac{1}{2}\log\big(x_1^2-x_1^2x_2^2\big)\Big]\\
&=\inf_{x_1x_2=y}\Big[-\frac{1}{2}\log\big(x_1^2-y^2\big)\Big]\\
&=-\frac{1}{2}\log(1-y^2)\,,
\end{align*}
since the infimum is attained at $x_1=1$ and $x_2=y$. If $y\notin[0,1)$, for every $(x_1,x_2)\in\R^2$ such that $x_1x_2=y$, we have $\mathcal{I}(x_1,x_2)=+\infty$.
\end{proof}

Finally, we consider the case $\lambda=1$, where the LDP for $\bV_1$ is a consequence of Lemma \ref{lem: lamda one without U}.

\begin{cor}\label{cor: rate funcion V_1 lambda=1}
For each $n\in\N$ let $k_n\in\N$ with $k_n\in\{1,\ldots,n-1\}$ be such that
\[
\lim_{n\to\infty}\frac{k_n}{n} = 1\,.
\]
Then, the sequence of random variables $\bV_1$ satisfies an LDP with speed $n$ and rate function
\[
\mathcal{I}_{\bV_1}^{(1)}(y):=\begin{cases}
-\frac{1}{2}\log(y^2) &: y\in(0,1]\cr
+\infty &: \text{otherwise}\,.
\end{cases}
\]
\end{cor}
\begin{proof}
By Proposition \ref{JointRateFunction}, Lemma \ref{lem: lamda one without U} and Lemma \ref{lem:gantert rate function U} the sequence of random vectors $\big((U^{1/n},V^{(n)})\big)_{n\in\N}$ satisfies an LDP with speed $n$ and rate function
\[
\mathcal{I}(x_1,x_2)=\begin{cases}
-\log (x_1) - \frac{1}{2}\log(x_2^2) &: x_1\in(0,1]\text{ and }x_2\in(0,1] \\
+\infty &: \text{otherwise}\,,
\end{cases}
\]
$(x_1,x_2)\in\R^2$.
Defining $F(x_1,x_2):=x_1x_2$ and applying the contraction principle (see Proposition \ref{prop:contraction principle}), we deduce that $\bV_1=\big(F(U^{1/n},V^{(n)})\big)_{n\in\N}$ satisfies an LDP with speed $n$ and rate function
\[
\mathcal{I}_{\bV_1}^{(1)}(y) = \inf_{(x_1,x_2)\in\R^2\atop F(x_1,x_2)=y}I(x_1,x_2)\,,\qquad y\in\R\,.
\]
If $y\in(0,1]$, then
\begin{align*}
\mathcal{I}_{\bV_1}^{(1)}(y) &= \inf_{x_1x_2=y}\Big[-\log (x_1) -\frac{1}{2}\log(x_2^2)\Big]\\
&=\inf_{x_1x_2=y}\Big[-\frac{1}{2}\log(x_1^2)-\frac{1}{2}\log(x_2^2)\Big]\\
&=\inf_{x_1x_2=y}\Big[-\frac{1}{2}\log(x_1^2x_2^2)\Big]\\
&=\inf_{x_1x_2=y}\Big[-\frac{1}{2}\log(y^2)\Big]\\
&=-\frac{1}{2}\log(y^2)\,.
\end{align*}
If $y\notin(0,1]$, for every $(x_1,x_2)\in\R^2$ such that $x_1x_2=y$, we have  $\mathcal{I}(x_1,x_2)=+\infty$.
\end{proof}

\subsection{LDP for the random sequence $\bW$}\label{subsec: LDP W^n}

In this subsection we prove LDPs for the sequence $\bW=(W^{(n)})_{n\in\N}$. We will only consider the case $p\in[2,\infty)$ (the special situation in which $p=\infty$ is treated directly in the proof of Theorem \ref{thm:p>2} in Section \ref{sec:proof main results}), where the result follows from Cram\'er's theorem and the contraction principle. 

\begin{lemma}\label{lem:rate function p greater 2}
Let $p\in[2,\infty)$.
Then $\bW$ satisfies an LDP with speed $n$ and rate function
\[
\mathcal{I}_\bW^{(p)}(y):=
\begin{cases}
\inf\limits_{x_1\geq 0,x_2>0\atop x_1^{1/2}x_2^{-1/p}=y}\mathcal{I}^*(x_1,x_2)&: y\geq 0\cr
+\infty &: y< 0\,,
\end{cases}
\]
where $\mathcal{I}^*$ is the Legendre-Fenchel transform of
$$
\mathcal{I}(t_1,t_2):=\log\left(\int_{\R}e^{t_1x^2+t_2\vert x\vert^p}\frac{e^{-\frac{\vert x\vert^p}{p}}}{2p^{1/p}\Gamma(1+\frac{1}{p})}\,\dif x\right)
$$
with effective domain $\R\times(-\infty,1/p)$ if $p>2$ and $\{ (t_1,t_2)\in\R^2:t_1+t_2 <\frac{1}{2}\}$ if $p=2$.
\end{lemma}

\begin{proof}
We set
\[
S^{(n)}:= \frac{1}{n}\sum_{i=1}^{n} (Z_i^2,\vert Z_i\vert^p)\,, \qquad n\in\N\,.
\]
Let $t=(t_1,t_2)\in\R^2$ and define
\begin{align*}
\mathcal{I}(t_1,t_2)& := \log\left(\E\, e^{\left\langle t, (Z_1^2,|Z_1|^p) \right\rangle} \right)\cr
&=\log\left(\int_{\R}e^{t_1x^2+t_2\vert x\vert^p}\frac{e^{-\frac{\vert x\vert^p}{p}}}{2p^{1/p}\Gamma(1+\frac{1}{p})}\,\dif x\right)\cr
&=\log\left(\int_0^\infty \frac{e^{\frac{1}{p}(pt_1x^2-(1-pt_2)x^p)}}{p^{1/p}\Gamma(1+\frac{1}{p})}\,\dif x\right)\,,
\end{align*}
which is finite in $\R\times\big(-\infty,\frac{1}{p}\big)$ if $p>2$ and if $t_1+t_2<\frac{1}{2}$ for $p=2$. Since $(0,0)$ is in the interior of the effective domain of $\mathcal{I}$, by Cram\'er's theorem (see Proposition \ref{prop:cramer}), $(S^{(n)})_{n\in\N}$ satisfies an LDP with speed $n$ and rate function $\mathcal{I}^*$. Notice that the effective domain of $\mathcal{I}^*$ is contained in $[0,\infty)\times[0,\infty)$. Moreover, $(W^{(n)})_{n\in\N}=(F(S^{(n)}))_{n\in\N}$, with $F\,:\,\R^2\to\R$ being the function given by
$$
F(x_1,x_2)={x_1^{1/2}}{x_2^{-1/p}}.
$$
Note that this function is continuous on $[0,\infty)\times(0,\infty)$. Hence, by the contraction principle (see Proposition \ref{prop:contraction principle}), $\bW$ satisfies an LDP with speed $n$ and rate function
$$
\mathcal{I}_\bW^{(p)}(y)=\inf_{x_1 \geq 0, x_2>0\atop {x_1^{1/2}}{x_2^{-1/p}}=y}\mathcal{I}^*(x_1,x_2)\,,
$$
if $y\geq 0$ and $\mathcal{I}_\bW^{(p)}(y)=+\infty$ if $y<0$, because $F^{-1}(\{y\})=\emptyset$.
\end{proof}

\begin{rmk}\label{rem:RateFunctionWp=2}
Note that if $p=2$, the random variables $W^{(n)}$, $n\geq 1$ are constantly equal to $1$. This means that for any $A\in\mathscr{B}(\R)$,
\[
\lim_{n\to\infty}\frac{\log\left(\Pro(W^{(n)}\in A)\right)}{n} =
\begin{cases}
0  &: 1\in A \cr
-\infty &: 1\notin A\,.
\end{cases}
\]
Therefore,
\[
\mathcal{I}_\bW^{(2)}(y)
=\begin{cases}
0 &: y=1\cr
+\infty &: y\neq 1\,.
\end{cases}
\]
\end{rmk}

\section{Proof of the main results}\label{sec:proof main results}

After these preparations, we can now present the proofs of our main results, Theorem \ref{thm:p>2} and Theorem \ref{thm:p<2}.

\subsection{Proof of Theorem \ref{thm:p>2}}

First, let $p\in[2,\infty)$. According to Theorem \ref{RepresentationAnnealed}, for each $n\in\N$, the random variable $n^{{1\over p}-{1\over 2}}\|P_{E^{(n)}}X^{(n)}\|_2$ has the same distribution as $V_1^{(n)}W^{(n)}$. By Proposition \ref{JointRateFunction}, if for $\lambda\in[0,1]$, $\mathcal{I}_{\bV_1}^{(\lambda)}$ and $\mathcal{I}_\bW^{(p)}$ are the rate functions defined in Corollaries \ref{cor:lambda 0,1 including U}, \ref{cor: rate function V_1 lambda=0} and \ref{cor: rate funcion V_1 lambda=1}, and Lemma \ref{lem:rate function p greater 2} the sequence of random vector $\big((V_1^{(n)},W^{(n)})\big)_{n\in\N}$ satisfies an LDP with speed $n$ and rate function
\begin{align*} &\mathcal{I}_{\bV_1}^{(\lambda)}(x_1)+\mathcal{I}_\bW^{(p)}(x_2) \\
&= \begin{cases}
{\lambda\over 2}\log\big({\lambda\over x_1^2}\big)+{1-\lambda\over 2}\log\big({1-\lambda\over 1-x_1^2}\big)+\mathcal{I}_{\bW}^{(p)}(x_2) &: x_1\in(0,1)\text{ and }x_2\in D_{\mathcal{I}_\bW^{(p)}} \\ +\infty &: \text{otherwise}\,,
\end{cases}
\end{align*}
$(x_1,x_2)\in\R^2$.
By the contraction principle (see Proposition \ref{prop:contraction principle}) applied to the function $F:\mathbb{R}^2\to\mathbb{R},\,F(x_1,x_2)= x_1x_2$, we conclude that the sequence of random variables $(F(V_1^{(n)},W^{(n)}))_{n\in\N}$ satisfies an LDP with speed $n$ and rate function
\[
\mathcal{I}_{\bPEX}(y)=\inf_{x_1x_2=y}\big[\mathcal{I}_{\bV_1}(x_1)+\mathcal{I}_\bW^{(p)}(x_2)\big]\,, \qquad y\in\R\,.
\]
If $y<0$, then, for any $x_1,x_2\in\R$ such that $x_1x_2=y$, either $x_1$ or $x_2$ is negative and so $\mathcal{I}_{\bPEX}(y)=+\infty$. If $y=0$, then, for any $x_1,x_2\in\R$ with $x_1x_2=y$, either $x_1=0$, or $x_2=0$. If $x_1=0$ then $\mathcal{I}_{\bV_1}^{(\lambda)}(0)=+\infty$ if $\lambda\neq 0$. If $\lambda=0$ then $\mathcal{I}_{\bV_1}^{(0)}(0)=0$. Thus, if $\lambda\neq0$, we see that
\[
\mathcal{I}_{\bPEX}(0) = \inf_{x_1\in(0,1)}\big[\mathcal{I}_{\bV_1}(x_1) + \mathcal{I}_\bW^{(p)}(0)\big]\,.
\]
Since $\inf\limits_{x_1\in(0,1)}\mathcal{I}_{\bV_1}(x_1)$ is attained when $x_1=\sqrt{\lambda}$ and $\mathcal{I}_{\bV_1}^{(\lambda)}(\sqrt{\lambda})=0$, we obtain $\mathcal{I}_{\bPEX}(0)=\mathcal{I}_\bW^{(p)}(0)$. If $\lambda=0$, then
$$
\mathcal{I}_{\bPEX}(0)=\min\{\mathcal{I}_\bW^{(p)}(0),\inf_{x\geq0}\mathcal{I}_\bW^{(p)}(x)\}=\inf_{x\geq0}\mathcal{I}_\bW^{(p)}(x)\,.
$$
If $y>0$, then, for any $x_1,x_2\in\R$ such that $x_1x_2=y$, we can write $x_1=\frac{y}{x_2}$. Since if $0<x_2< y$, we have that $x_1>1$, in such a case $\mathcal{I}_{\bV_1}^{(\lambda)}(x_1)=+\infty$. If $x_2<0$ then $\mathcal{I}_\bW^{(p)}(x_2)=+\infty$. Thus,
\begin{align*}
\mathcal{I}_{\bPEX}(y) &= \inf_{x_2\geq y}\Big[{\lambda\over 2x_2^2}\log\Big({\lambda\over y^2}\Big)+{1-\lambda\over 2}\log\Big({1-\lambda\over 1-\frac{y^2}{x_2^2}}\Big)+\mathcal{I}_\bW^{(p)}(x_2)\Big]\,,
\end{align*}
which is the function in the statement of the theorem.

Finally, we consider the case $p=+\infty$ and notice that ${1\over \sqrt{n}}\|P_EX\|_2$ has the same distribution as the product $(\overline{W}^{(n)})^{1/2}\,V^{(n)}$ with $\overline{W}^{(n)}:={1\over n}\sum\limits_{i=1}^nX_i^2$, where $X=(X_1,\ldots,X_n)$ is a random vector whose entries are independent and uniformly distributed on $[-1,1]$. By Cram\'er's theorem (see Proposition \ref{prop:cramer}) it follows that $\overline{W}^{(n)}$ satisfies an LDP with speed $n$ and rate function $\mathcal{I}_\infty^*$, the Legendre-Fenchel transform of $\mathcal{I}_\infty(t)=\log\big(2\int_0^1 e^{tx^2}\,\dint t\big)$. Thus, according to Lemma \ref{lem: lamda one without U} and the contraction principle (see Proposition \ref{prop:contraction principle}), the sequence $({1\over \sqrt{n}}\|P_EX\|_2)_{n\in\N}$ satisfies an LDP with speed $n$ and rate function
\begin{align*}
\mathcal{I}_{\bPEX}(y) & = \inf_{x_1,x_2\in\R\atop x_1^{1/2}x_2=y}[\mathcal{I}_\infty(x_1)+\mathcal{I}_\bV^{(\lambda)}(x_2)] \\
&=\begin{cases}
\inf\limits_{x\geq y}\Big[{\lambda\over 2}\log\Big({\lambda x^2\over y^2}\Big)+{1-\lambda\over 2}\log\Big({1-\lambda\over 1-y^2x^{-2}}\Big)\Big] &: y>0\\
\inf\limits_{x> 0} \mathcal{I}_\infty(x)&:y=0\text{ and }\lambda=0\\
+\infty &: \text{otherwise}\,.
\end{cases}
\end{align*}
This complete the proof of the theorem. \hfill $\Box$

\begin{rmk}
In the special case $p=2$ we find that
\begin{align*}
\mathcal{I}_{\bPEX}(y) = \inf_{x_1x_2=y}\big[\mathcal{I}_{\bV_1}^{(\lambda)}(x_1)+\mathcal{I}_{\bW}^{(2)}(x_2)\big]=\mathcal{I}_{\bV_1}^{(\lambda)}(y)+\mathcal{I}_{\bW}^{(2)}(1)=\mathcal{I}_{\bV_1}(y)\,,
\end{align*}
since $\mathcal{I}_{\bW}^{(2)}(1)=0$, see Remark \ref{rem:RateFunctionWp=2}. This proves relation \eqref{eq:RateFunctionp=2} in the introduction.
\end{rmk}

\subsection{Proof of Theorem \ref{thm:p<2}}

The proof of Theorem \ref{thm:p<2} requires different tools. In particular, it relies on a large deviation result for sums of so-called stretched exponential random variables taken from a paper of Gantert, Ramanan and Rembart \cite{GantertRamananRembart}. We start by computing the variance of a $p$-generalized Gaussian random variable.

\begin{lemma}\label{lem:ConstantM}
Let $Z$ be a $p$-generalized Gaussian random variable for some $p\in[1,\infty)$. Then
$$
\E Z^2 = {p^{p\over 2}\over 3}{\Gamma\big(1+{3\over p}\big)\over\Gamma\big(1+{1\over p}\big)}\,.
$$
\end{lemma}
\begin{proof}
Recalling the definition of the density $f_p$ of $Z$ from Proposition \ref{prop:schechtman zinn} and applying the change of variables $u=x^p/p$, we see that
\begin{align*}
\E Z^2 = \int_{-\infty}^\infty x^2\,f_p(x)\,\dint x = {1\over p^{1-{2\over p}}\Gamma\big(1+{1\over p}\big)}\int_0^\infty u^{{3\over p}-1}\,e^{-u}\,\dint u = {p^{p\over 2}\over 3}{\Gamma\big(1+{3\over p}\big)\over\Gamma\big(1+{1\over p}\big)}
\end{align*}
and the proof is complete.
\end{proof}

The next lemma provides bounds for the tails of the random variable $Z^2$. A function $f:(0,\infty)\to\R$ is said to be slowly varying (at infinity) provided that $\lim\limits_{t\to\infty}{f(at)\over f(t)}=1$ for any $a>0$.

\begin{lemma}\label{lem:Tailsp<2}
For $p\in[1,2)$ let $Z$ be a $p$-generalized Gaussian random variable and for $t>0$ define functions
\begin{align*}
b(t) := {1\over p}+{p-1\over 2}t^{-{p\over 2}}\log t\,,\qquad c_1(t) := {t^{p\over 2}\over t^{p\over 2}+1}\qquad\text{and}\qquad c_2=c_2(t):=2.
\end{align*}
These functions are slowly varying and, for all $t>0$, one has that
$$
c_1(t)\,e^{-b(t)\,t^{p\over 2}}\leq \Pro(Z^2\geq t)\leq c_2\,e^{-b(t)\,t^{p\over 2}}\,.
$$
\end{lemma}
\begin{proof}
Let us first check that $b,c_1$ and $c_2$ are slowly varying. For $c_2$ this is trivial, while for $b$ and $c_1$ we have that, for all $a>0$,
\begin{align*}
\lim_{t\to\infty}{b(at)\over b(t)} &= \lim_{t\to\infty}{{1\over p}+{p-1\over 2}(at)^{-{p\over 2}}\log(a t)\over {1\over p}+{p-1\over 2}t^{-{p\over 2}}\log t} = 1\,,\\
\lim_{t\to\infty}{c_1(a t)\over c_1(t)} &= \lim_{t\to\infty}{(a t)^{p\over 2}\over (a t)^{p\over 2}+1}{t^{p\over 2}+1\over t^{p\over 2}} = 1\,.
\end{align*}
It is well known and easily shown that, for $t>0$,
$$
{t\over t^p+1}\,e^{-t^p/p}\leq\int_t^\infty e^{-s^p/p}\,\dint s\leq {1\over t^{p-1}}\,e^{-t^p/p}\,.
$$
This readily implies the upper bound
$$
\Pro(Z^2\geq t)\leq {2\over t^{p-1\over 2}}\,e^{-t^{p\over 2}/p} = 2\,e^{-t^{p\over 2}\big({1\over p}+{p-1\over 2}t^{-{p\over 2}}\log t\big)} = c_2\,e^{-b(t)\,t^{p\over 2}}
$$
as well as the lower bound by writing
\begin{align*}
\Pro(Z^2\geq t)\geq {2\sqrt{t}\over t^{p\over 2}+1}\,e^{-t^{p\over 2}/p} = {t^{p\over 2}\over t^{p\over 2}+1}\,e^{-t^{p\over 2}\big({1\over p}+{p-1\over 2}t^{-{p\over 2}}\log t\big)} = c_1(t)\,e^{-b(t)\,t^{p\over 2}}\,.
\end{align*}
The argument is thus complete.
\end{proof}

In the arguments that follow we also need the following two auxiliary LDP's.

\begin{lemma}\label{lem:AuxLDPsp<2}
\begin{itemize}
\item[(a)] Suppose that $k_n\to\infty$, as $n\to\infty$. Then the sequence $\bG:=({1\over k_n}\sum_{i=1}^{k_n}g_i^2)_{i\in\N}$ satisfies an LDP with speed $k_n$ and rate function
$$
\mathcal{I}_{\bG}(y) = \begin{cases}
{y-1\over 2}-{1\over 2}\log y &: y>0\\
+\infty &: \text{otherwise}\,.
\end{cases}
$$
\item[(b)] Let $p\in[1,\infty)$. Then the sequence $\bZ_p:=({1\over n}\sum_{i=1}^{n}|Z_i|^p)_{i\in\N}$ satisfies an LDP with speed $n$ and rate function
$$
\mathcal{I}_{\bZ_p}(y) = \begin{cases}
\frac{1}{p}y- y^{1\over p+1}\big(1+\frac{1}{p}\big) &: y>0\\
+\infty &: \text{otherwise}\,.
\end{cases}
$$
\end{itemize}
\end{lemma}
\begin{proof}
Part (a) has already been verified in the proof of Lemma \ref{lem:GaussianPartLambdaIn(0,1)}. To prove the statement in $(b)$ we apply Cram\'er's theorem (Proposition \ref{prop:cramer}). Indeed, the moment generating function of $|Z|^p$, where $Z$ has a $p$-generalized Gaussian distribution, is given by
\begin{align*}
\mathcal{I}(y) = \int_{-\infty}^\infty e^{|x|^py}\,f_p(x)\,\dint x = {1\over (1-yp)^{1\over p}}\,,\qquad y<{1\over p}\,.
\end{align*}
In particular, zero is an interior point of the effective domain of $\mathcal{I}$. As a consequence, ${\bf Z}_p$ satisfies an LDP with speed $n$ and rate function given by the Legendre-Fenchel transform $\mathcal{I}^*$ of $\mathcal{I}$. The latter is given by
\begin{align*}
\mathcal{I}^*(y) = \sup_{x\in\R}[xy-{\mathcal I}(x)] =\frac{1}{p}y- y^{1\over p+1}\Big(1+\frac{1}{p}\Big)
\end{align*}
if $y>0$ and $\mathcal{I}^*(y) = +\infty$ otherwise.
\end{proof}

After these preparations, we can now present the proof of Theorem \ref{thm:p<2}. From now on we shall assume that we are dealing with a fixed parameter $p\in[1,2)$.

\begin{proof}[Proof of Theorem \ref{thm:p<2}]
Applying \cite[Theorem 1]{GantertRamananRembart} (with $a_j(n)={1\over n}$, $s=s_1=1$) together with Lemma \ref{lem:Tailsp<2} implies that for $y\geq m$, with $m=\E Z^2$ from Lemma \ref{lem:ConstantM},
$$
\lim_{n\to\infty}{1\over b(n)n^{p\over 2}}\log\Pro\Big({1\over n}\sum_{i=1}^nZ_i^2\geq y\Big) = -(y-m)^{p\over 2}\,.
$$
Since the sequences $(b(n)n^{p/ 2})_{n\in\N}$ and $({1\over p}n^{p/ 2})_{n\in\N}$ are asymptotically equivalent, the pre-factor $1/( b(n)n^{p/ 2})$ can be replaced by $1/n^{p/ 2}$. Moreover, since the random variables $Z_i^2$ are non-negative, this result can be lifted to an LDP, see also \cite[Remark 3.2]{GantertRamananRembart}. Thus, the sequence of random variables $\bZ:=({1\over n}\sum_{i=1}^n Z_i^2)_{n\in\N}$ satisfies an LDP with speed $n^{p/ 2}$ and rate function
$$
\mathcal{I}_{\bZ}(y) = \begin{cases}
{1\over p}(y-m)^{p\over 2} &: y\geq m\\
+\infty &: \text{otherwise}\,.
\end{cases}
$$
Next, we apply the contraction principle to the function $F$ given by $F:(0,\infty)\to(0,\infty),\,F(x)=\sqrt{x}$. This yields an LDP for the sequence $\sqrt{\bZ}:=(({1\over n}\sum_{i=1}^n Z_i^2)^{1\over 2})_{n\in\N}$ with speed $n^{p/ 2}$ and rate function
$$
\mathcal{I}_{\sqrt{\bZ}}(y) = \begin{cases}
{1\over p}(y^2-m)^{p\over 2} &: y\geq \sqrt{m}\\
+\infty &: \text{otherwise}\,.
\end{cases}
$$
In a next step, we apply Proposition \ref{prop:refinement contraction principle} to the functions $F_n:\R\to\R,\,F_n(x)=\sqrt{{k_n\over n}}\,x$ and $F:\R\to\R,\,F(x)=\sqrt{\lambda}\,x$, where $\lambda>0$ by assumption (the technical condition in Proposition \ref{prop:refinement contraction principle} is easily seen to be satisfied in this situation). This leads to an LDP for the sequence $\widetilde{\bZ}:=(\sqrt{k_n\over n}({1\over n}\sum_{i=1}^nZ_i^2))_{n\in\N}$ with speed $n^{p/ 2}$ and rate function
\begin{equation}\label{eq:RateZTilde}
\mathcal{I}_{\widetilde{\bZ}}(y) = \inf_{F(x)=y}\mathcal{I}_{\sqrt{\bZ}}(x) = \begin{cases}
{1\over p}({y^2\over\lambda}-m)^{p\over 2} &: y\geq \sqrt{\lambda m}\\
+\infty &: \text{otherwise}\,,
\end{cases}
\end{equation}
which coincides with the function $\mathcal{I}_{\bPEX}(y)$ in the statement of Theorem \ref{thm:p<2}.

In the remaining part of the proof we shall argue that the two random sequences $\widetilde{\bZ}$ and $\bPEX$ are exponentially equivalent and thus satisfy the same LDP. For this, we observe that according to Theorem \ref{RepresentationAnnealed}, for each $n\in\N$, the random variable $\|n^{{1\over p}-{1\over 2}}P_{E^{(n)}}X^{(n)}\|$ has the same distribution as
$$
U^{1/n}{\big({1\over n}\sum\limits_{i=1}^nZ_i^2\big)^{1/2}\over \big({1\over n}\sum\limits_{i=1}^n|Z_i|^p\big)^{1/p}}\,{\big({1\over k_n}\sum\limits_{i=1}^{k_n}g_i^2\big)^{1/2}\over\big({1\over n}\sum\limits_{i=1}^ng_i^2\big)^{1/2}}\,\sqrt{k_n\over n}\,.
$$
Now, fix $\delta,\varepsilon>0$ and note that
\begin{align*}
&\Pro\Bigg(\Bigg|\sqrt{k_n\over n}\Big({1\over n}\sum_{i=1}^nZ_i^2\Big)^{1/2}-U^{1/n}{\big({1\over n}\sum\limits_{i=1}^nZ_i^2\big)^{1/2}\over \big({1\over n}\sum\limits_{i=1}^n|Z_i|^p\big)^{1/p}}\,{\big({1\over k_n}\sum\limits_{i=1}^{k_n}g_i^2\big)^{1/2}\over\big({1\over n}\sum\limits_{i=1}^ng_i^2\big)^{1/2}}\,\sqrt{k_n\over n}\Bigg|>\delta\Bigg)\\
&\leq \Pro\Bigg(\sqrt{k_n\over n}\Big({1\over n}\sum_{i=1}^nZ_i^2\Big)^{1/2}>{\delta\over\varepsilon}\Bigg)\\
&\qquad+\Pro\Bigg(1-U^{1/n}{1\over \big({1\over n}\sum\limits_{i=1}^n|Z_i|^p\big)^{1/p}}\,{\big({1\over k_n}\sum\limits_{i=1}^{k_n}g_i^2\big)^{1/2}\over\big({1\over n}\sum\limits_{i=1}^ng_i^2\big)^{1/2}}>\varepsilon\Bigg)\\
&\qquad+\Pro\Bigg(1-U^{1/n}{1\over \big({1\over n}\sum\limits_{i=1}^n|Z_i|^p\big)^{1/p}}\,{\big({1\over k_n}\sum\limits_{i=1}^{k_n}g_i^2\big)^{1/2}\over\big({1\over n}\sum\limits_{i=1}^ng_i^2\big)^{1/2}}<-\varepsilon\Bigg)\\
&=:T_1+T_2+T_3\,.
\end{align*}
We further estimate $T_2$ by
\begin{align*}
T_2 &\leq \Pro(U^{1/n}<(1-\varepsilon)^{1/4})+\Pro\Big({1\over k_n}\sum_{i=1}^{k_n}g_i^2<(1-\varepsilon)^{1/2}\Big)\\
&\qquad+\Pro\Big({1\over n}\sum_{i=1}^n|Z_i|^p>(1-\varepsilon)^{-p/4}\Big)+\Pro\Big({1\over n}\sum_{i=1}^ng_i^2>(1-\varepsilon)^{-1/2}\Big)\\
&=:T_{2,1}+T_{2,2}+T_{2,3}+T_{2,4}\,.
\end{align*}
According to Lemma \ref{lem:gantert rate function U} and Lemma \ref{lem:AuxLDPsp<2}, the terms $T_{2,1}$, $T_{2,3}$ and $T_{2,4}$ decay exponentially with speed $n$. Indeed, this follows from the fact that the rate functions of the corresponding LDP's do not vanish at $(1-\varepsilon)^{1/4}$, $(1-\varepsilon)^{-p/4}$ and $(1-\varepsilon)^{-1/2}$, respectively. In addition and again by Lemma \ref{lem:AuxLDPsp<2}, the term $T_{2,2}$ decays exponentially with speed $k_n$ and again the rate function in the corresponding LDP does vanish at $(1-\varepsilon)^{-p/4}$.

Similarly, for $T_3$ we have the bound
\begin{align*}
T_3 &\leq \Pro(U^{1/n}>(1+\varepsilon)^{1/4}) + \Pro\Big({1\over k_n}\sum_{i=1}^{k_n}g_i^2>(1+\varepsilon)^{1/2}\Big)\\
&\qquad +\Pro\Big({1\over n}\sum_{i=1}^n|Z_i|^p<(1+\varepsilon)^{-p/4}\Big)+\Pro\Big({1\over n}\sum_{i=1}^ng_i^2<(1+\varepsilon)^{-1/2}\Big)\\
&=\Pro\Big({1\over k_n}\sum_{i=1}^{k_n}g_i^2>(1+\varepsilon)^{1/2}\Big)\\
&\qquad +\Pro\Big({1\over n}\sum_{i=1}^n|Z_i|^p<(1+\varepsilon)^{-p/4}\Big)+\Pro\Big({1\over n}\sum_{i=1}^ng_i^2<(1+\varepsilon)^{-1/2}\Big)\,.
\end{align*}
As for $T_2$ discussed above, these terms decay exponentially with speed $n$ and since ${k_n\over n}\to \lambda>0$, as $n\to\infty$, we conclude that
\begin{align*}
\limsup_{n\to\infty}{1\over n^{p/ 2}}\log T_2+\limsup_{n\to\infty}{1\over n^{p/ 2}}\log T_3 = -\infty\,.
\end{align*}
Hence,
\begin{align*}
&\limsup_{n\to\infty}{1\over n^{p/ 2}}\log\Pro\Big(\Big|\sqrt{k_n\over n}\Big({1\over n}\sum_{i=1}^nZ_i^2\Big)^{1/2}-\|n^{{1\over p}-{1\over 2}}P_{E^{(n)}}X^{(n)}\|\Big|>\delta\Big)\\
&\qquad\leq\limsup_{n\to\infty}{1\over n^{p/ 2}}\log T_1.
\end{align*}
Sending $\varepsilon\to 0$, the above LDP for the sequence $\widetilde{\bZ}$ (recall \eqref{eq:RateZTilde}) shows that this limit exists and is equal to $-\infty$. We have thus proved that the random sequences $\widetilde{\bZ}$ and $\bPEX$ are exponentially equivalent. So, by Proposition \ref{prop:exponentially equivalent} they satisfy the same LDP. This completes the proof of Theorem \ref{thm:p<2}.
\end{proof}

\section{Appendix}

Let us present the proof of Proposition \ref{JointRateFunction}.

\begin{proof}
For every $m\in\N$, any $\delta>0$ and any $x\in\R^m$, let us denote by
\[
B_{\infty}(x,\delta):=\{y\in\R^m\,:\,\Vert x-y\Vert_\infty<\delta\}\,,
\]
the cube in $\R^m$ centered at $x$ with side length $2\delta$.
Let $d_1,d_2\in\N$.

\subsubsection*{Lower bound} Let $A\in\mathscr{L}(\R^{d_1}\times\R^{d_2})$ with non-empty interior. Let $z=(x,y)\in A^\circ$ and $\delta>0$ such that $B_{\infty}(z,\delta)=B_\infty(x,\delta)\times B_\infty(y,\delta)\subset A$.
From the independence of $X^{(n)}$ and $Y^{(n)}$, we conclude that
\begin{align*}
\frac{1}{s(n)}\log\Big(\Pro\big(Z^{(n)}\in A^\circ\big)\Big)
& \geq \frac{1}{s(n)}\log\Big(\Pro\big(Z^{(n)}\in B_\infty(x,\delta)\times B_\infty(y,\delta)\big)\Big) \cr
&= \frac{1}{s(n)}\log\Big(\Pro\big(X^{(n)}\in B_\infty(x,\delta)\big)\Big) \cr
&\qquad\qquad+\frac{1}{s(n)}\log\Big(\Pro\big(Y^{(n)}\in B_\infty(y,\delta)\big)\Big).
\end{align*}
Consequently, for every $z=(x,y)\in A^\circ$ and $\delta>0$ such that $B_{\infty}(z,\delta)\subset A^\circ$,
$$
 \liminf_{n\to\infty}\frac{1}{s(n)}\log\Big(\Pro\big(Z^{(n)}\in A^\circ\big)\Big)\geq-\inf_{y_1\in B_\infty(x,\delta)}\mathcal I_\bX(y_1)-\inf_{y_2\in B_\infty(y,\delta)}\mathcal I_\bY(y_2)
$$
and, since this inequality is true for every such $x$ and $\delta>0$,
$$
 \liminf_{n\to\infty}\frac{1}{s(n)}\log\Big(\Pro\big(Z^{(n)}\in A^\circ\big)\Big)\geq-\inf_{(x,y)\in A^\circ}\big[\mathcal I_\bX(x)+\mathcal I_\bY(y)\big]\,.
$$

\subsubsection*{Upper bound} Since we are considering the  $\sigma$-algebra $\mathscr{L}(\R^{d_1}\times\R^{d_2})$, by Proposition \ref{prop:equivalence weak and full LDP} it is enough to prove the upper bound for subsets $A\subset
\R^{d_1}\times\R^{d_2}$ such that $\bar{A}$ is compact together with the exponential tightness for $\bZ$. In such a case, for every cover of $\bar{A}$ by subsets of the form $B_{\infty}(z_i,\delta_i)$, $z_i=(x_{i},y_{i})\in\R^{d_1}\times\R^{d_2}$ we can extract a finite number of sets such that  $\bar{A}\subset\bigcup\limits_{i=1}^NB_{\infty}(z_i,\delta_i)$ and for every such finite cover we have
\begin{align*}
&\frac{1}{s(n)}\log\Big(\Pro\big(Z^{(n)}\in \bar{A}\big)\Big)
 \leq \frac{1}{s(n)}\log\Big(\sum_{i=1}^N\Pro\big(Z^{(n)}\in B_\infty(x_{i},\delta_i)\times B_\infty(y_{i},\delta_i)\big)\Big) \cr
& \leq \frac{1}{s(n)}\log\Big(N\max_{1\leq i\leq N}\Pro\big(Z^{(n)}\in B_\infty(x_{i},\delta_i)\times B_\infty(y_{i},\delta_i)\big)\Big) \cr
&= \frac{\log N}{s(n)}
+ \max_{1\leq i\leq N}\left\{\frac{1}{s(n)}\log\Big(\Pro\big(X^{(n)}\in B_\infty(x_{i},\delta_i)\big)\Big)\right.\cr
&\qquad\qquad+\left.\frac{1}{s(n)}\log\Big(\Pro\big(Y^{(n)}\in B_\infty(y_{i},\delta_i)\big)\Big)\right\}.
\end{align*}
Now, taking the $\limsup\limits_{n\to\infty}$, since $N$ is a fixed number depending on the cover of $\bar{A}$ we extracted,  the first term tends to 0 and, since the $ \limsup\limits_{n\to\infty}$ of the maximum equals the maximum of the $ \limsup\limits_{n\to\infty}$, we have
\begin{align*}
&\limsup_{n\to\infty}\frac{1}{s(n)}\log\Big(\Pro\big(Z^{(n)}\in \bar{A}\big)\Big) \cr
 &\leq \max_{1\leq i\leq N}\left\{-\inf_{x\in B_\infty(x_{i},\delta_i)}\mathcal I_\bX(x)-\inf_{y\in B_\infty(y_{i},\delta_i)}\mathcal I_\bY(y)\right\}\cr
  & \qquad\qquad-\min_{1\leq i\leq N}\left\{\inf_{x\in B_\infty(x_{i},\delta_i)}\mathcal I_\bX(x)+\inf_{y\in B_\infty(y_{i},\delta_i)}\mathcal I_\bY(y)\right\}.
\end{align*}
Since this is true for any cover of $\bar{A}$ and every finite cover we extract from it, we obtain
$$
\limsup_{n\to\infty}\frac{1}{s(n)}\log\Big(\Pro\big(Z^{(n)}\in \bar{A}\big)\Big)\leq -\inf_{(x,y)\in \bar{A}}\big[\mathcal I_\bX(x)+\mathcal I_\bY(y)\big]\,,
$$
which proves the upper bound for subsets of $\R^{d_1}\times\R^{d_2}$ such that $\bar{A}$ is compact.

\subsubsection*{Exponential tightness} To show the exponential tightness of $\bZ$, let $\alpha>0$ be any positive number. Since $\bX$ and $\bY$ are assumed to satisfy a (full) LDP, by Proposition \ref{prop:equivalence weak and full LDP}, $\bX$ and $\bY$ are exponentially tight and, thus, there exist compact sets $K_{1,\alpha}\subset \R^{d_1}$ and $K_{2,\alpha}\subset\R^{d_2}$ such that
\[
\limsup_{n\to\infty}{1\over s(n)}\log\Pro(X^{(n)}\notin K_{1,\alpha})<-\frac{\alpha}{2}
\]
and
\[
\limsup_{n\to\infty}{1\over s(n)}\log\Pro(Y^{(n)}\notin K_{2,\alpha})<-\frac{\alpha}{2}\,.
\]
Thus, for any $\alpha>0$, taking $K_\alpha:=K_{1,\alpha}\times K_{2,\alpha}\subset\R^{d_1}\times\R^{d_2}$, we obtain that
\[
\limsup_{n\to\infty}{1\over s(n)}\log\Pro(Z^{(n)}\notin K_{\alpha})<-\alpha.
\]
Consequently, $\bZ$ is exponentially tight and the proof is thus complete.
\end{proof}

\section*{Acknowledgement}
The authors David Alonso-Guti\'errez and Joscha Prochno would like to thank the DFG Research Training Group GRK2131 for the financial support that made possible a research stay at Ruhr University Bochum. David Alonso-Guti\'errez is also supported by Instituto Universitario de Matemáticas y Aplicaciones (IUMA), Spanish Ministry of Sciences and Innovation (MICINN) project MTM2016-77710-P.

\bibliographystyle{plain}
\bibliography{ldp}

\end{document}